%% file: twists2.tex
\documentclass[11pt,a4paper]{amsart}

\usepackage{praeambel}
\usepackage{ap_makros}

\let\stdthebibliography\thebibliography
\let\stdendthebibliography\endthebibliography
\renewenvironment*{thebibliography}[1]{%
  \stdthebibliography{WWXX}}
  {\stdendthebibliography}

\title{Scattering Theory with Unitary Twists}


\author[M.~Doll]{Moritz Doll}
\address{Moritz Doll, University of Bremen, Department~3 -- Mathematics, 
Bibliothekstr.~5, 28359 Bremen, Germany}
\email{doll@uni-bremen.de}
\author[K.~Fedosova]{Ksenia Fedosova}
\address{Ksenia Fedosova, Albert Ludwigs University of Freiburg, Mathematical Institute, Ernst-Zermelo-Str. 1, 79104 Freiburg im Breisgau, Germany}
\email{ksenia.fedosova@math.uni-freiburg.de}
\author[A.~Pohl]{Anke Pohl}
\address{Anke Pohl, University of Bremen, Department~3 -- Mathematics, 
Bibliothekstr.~5, 28359 Bremen, Germany}
\email{apohl@uni-bremen.de}

\begin{document}

\begin{abstract}
We study the spectral properties of the Laplace operator associated to a hyperbolic surface
in the presence of a unitary representation of the fundamental group. 
Following the approach by Guillop\'e and Zworski, we establish a factorization 
formula for the twisted scattering determinant and describe the behavior 
of the scattering matrix in a neighborhood of~$1/2$.
\end{abstract}

\maketitle

\tableofcontents

\input{twists2_main}

\appendix

\bibliography{ap_bib} 
\bibliographystyle{amsalpha}

\setlength{\parindent}{0pt}

\end{document}

%% file: twists2_main.tex
\section{Introduction}

We consider a finitely generated Fuchsian group $\group \subset \PSL(2,\R)$
and denote the associated hyperbolic surface by~$X$. Thus $X=\group 
\bs \h$, where $\h$ denotes the hyperbolic upper half-plane and 
$\PSL(2,\R)$ acts via Möbius transformations on $\h$. Throughout this article, 
we will suppose that $X$ is non-elementary, geometrically finite and of 
infinite volume. However, we allow that $X$ has orbifold singularities or, 
equivalently, that $\group$ has torsion. 

We further consider a finite-dimensional unitary representation 
\[
\twist\colon \group \to \Unit(V)
\]
on a Hermitian vector space~$V$. The representation $\twist$ induces a Hermitian 
vector orbibundle 
\[
\bundle \coloneqq \Gamma \bs (\h \times V) \to X
\]
with typical fiber $V$. 
It is well-known that the (smooth) sections of $\bundle$ are in bijection with 
the smooth functions $f \colon \h \to V$ that obey the \emph{twisting} 
equivariance 
\begin{align}\label{eq:twisted}
    f(g.z) = \twist(g) f(z)\,, \quad z \in \h, \, g \in \group\,.
\end{align}
See, for example, \cite[Lemma 3.3]{DFP} for details. On smooth maps $f \colon 
\h \to V$, the hyperbolic Laplacian is given by
\begin{align*}
    \Delta_{\h} f(z) = - \sum_{j=1}^{\dim V} y^2 \left( \frac{\pa^2}{\pa x^2} + \frac{\pa^2}{\pa y^2} \right) f(z)\,,
\end{align*}
where $z = x + iy \in \h$.
Using the identification of twisted functions (see~\eqref{eq:twisted}) and 
sections of $\bundle$ and the fact that $\twist$ is unitary,
the Laplacian $\Delta_{\h}$ gives rise to a non-negative self-adjoint operator 
\[
\LapTwist \colon L^2(X,\bundle) \to L^2(X,\bundle)\,.
\]
For $\Re s > 1/2$ and $s \not \in [1/2,1]$, the resolvent of~$\LapTwist$ is 
defined by
\begin{align*}
    \ResTwist(s) \coloneqq (\LapTwist - s(1-s))^{-1} \colon  L^2(X,\bundle) \to L^2(X,\bundle)\,.
\end{align*}
As shown in~\cite[Theorem A]{DFP}, the resolvent~$\ResTwist$ admits a 
meromorphic continuation to $s \in \C$ as an operator
\begin{align*}
    \ResTwist(s)\colon L_\cpt^2(X,\bundle) \to L^2_\loc(X,\bundle)\,.
\end{align*}
The poles of $\ResTwist(s)$ are the \emph{resonances} of $\LapTwist$. The 
\emph{multiplicity} of the pole~$s\in \C$ is the rank of the residue at~$s$. 

In~\cite[Theorem B]{DFP}, we showed that the resonance counting function grows 
at most quadratically, i.e.,
\begin{align*}
\sum_{\substack{s \in \ResSet_{X,\twist}\\\abs{s} \leq r}} m_{X,\twist}(s) = O(r^2)\qquad\text{as $r\to\infty$}\,,
\end{align*}
where $\ResSet_{X,\twist}$ denotes the set of resonances and $m_{X,\twist}(s)$ 
the multiplicity of~$s\in \ResSet_{X,\twist}$. Hence, by the Weierstrass 
factorization theorem, there exists an entire function, $\ProdTwist$, such that 
its zeros coincide with the resonances, and the multiplicity of a zero $s$ of 
$\ProdTwist$ is equal to $m_{X,\twist}(s)$.
We also define the  Weierstrass product $\ProdModel(s)$ associated to the 
resonances of the disjoint union of funnel ends $X_f$ (see 
Section~\ref{sec:relative_scattering_matrix} for details).

We consider the \emph{scattering matrix}, which is a certain operator
\begin{align*}
S_{X,\twist}(s)\colon \CI(\pa_\infty X, \bundle) \to \CI(\pa_\infty X, \bundle), \quad s \not \in \ResSet_{X,\twist} \cup \Z/2,
\end{align*}
defined on the boundary of a suitable compactification of $\bundle$ (see 
Sections~\ref{sec:scattering_matrix_for_model} 
and~\ref{sec:scattering-determinant}).
For each $\psi \in \CI(\pa_\infty X, \bundle)$ there exists $u \in \CI(X, \bundle)$ such that
$(\LapTwist - s(1-s)) u = 0$ and
\begin{align*}
    (2s-1) u \sim \rho_f^{1-s} \rho_c^{-s} \psi + \rho_f^s \rho_c^{s-1} S_{X,\twist}(s) \psi   \qquad\text{as $\rho_f\rho_c\to0$}\,,
\end{align*}
where $\rho_f$ and $\rho_c$ are the boundary defining 
functions in the funnel and cusp ends, respectively.
Even though the scattering matrix is not trace class, we can define a regularized determinant of $S_{X,\twist}(s)$, that we will call the \emph{relative scattering determinant},~$\tau_{X,\twist}(s)$.

As the first main result of this article, we prove a factorization of the 
relative scattering determinant in terms of the Weierstrass product over the 
resonances.

\begin{theorem}\label{thm:scattering-determinant}
    The scattering determinant admits the factorization
    \begin{align*}
        \tau_{X,\twist}(s) = e^{q(s)} \frac{\ProdTwist(1-s)}{\ProdTwist(s)} 
        \frac{\ProdModel(s)}{\ProdModel(1-s)}\,,
    \end{align*}
    where $q\colon \C \to \C$ is a polynomial of degree at most $4$.
\end{theorem}
For $\dim V = 1$ and $\twist=\id$, Theorem~\ref{thm:scattering-determinant} 
reduces to~\cite[Proposition~3.7]{GuZw97}. This latter result plays a crucial 
role in the proof of the factorization of the Selberg zeta function by 
Borthwick--Judge--Perry~\cite{BJP}.

We remark that Theorem~\ref{thm:scattering-determinant} implies that the 
scattering determinant has no pole or zero at $s = 1/2$. However, $s=1/2$ might 
be a resonance. The second main result of this article shows that we are 
able to describe the behavior of the scattering matrix~$S_{X,\twist}(s)$ in 
some (small) neighborhood of~$1/2$. For this, we set
\begin{align}\label{eq:P_at_1/2}
    P \coloneqq \frac12\left( S_{X,\twist}\bigl(\tfrac12\bigr) + \id \right)\,.
\end{align}
Then  
\[
S_{X,\twist}(s) = -\id + 2P + (2s-1) T_{X,\twist}(s)
\]
with $T_{X,\twist}$ being an operator family that is holomorphic in a small 
neighborhood of $s=1/2$.

\begin{theorem}\label{thm:smatrix-onehalf}
    The operator $P$ is an orthogonal projection of rank $m_{X,\twist}(1/2)$ onto the space of elements in~$\CI(\pa_\infty X, \bundle)$
    that are invariant under the map $S_{X,\twist}(1/2)$.
\end{theorem}

\subsubsection*{Structure of this article} 
In Section~\ref{sec:scattering_matrix_for_model}, we discuss the scattering 
matrices for the model funnel and the parabolic cylinder. In Section 
\ref{sec:resolvent}, we obtain a decomposition of the resolvent, study the 
structure of the resolvent close to a resonance and obtain that there are no 
resonances on the line $\Re(s)=1/2$ except for, maybe, $s = 1/2$. 
In Section~\ref{sec:scattering-determinant}, we  introduce the scattering 
matrix, the relative scattering determinant and prove 
Theorems~\ref{thm:scattering-determinant} and~\ref{thm:smatrix-onehalf}.

\subsubsection*{Acknowledgements}
AP's research is funded by the Deutsche Forschungsgemeinschaft (DFG, German Research Foundation) -- project no.~441868048 (Priority Program~2026 ``Geometry at Infinity''). MD was partially funded by a Universit\"at Bremen ZF 04-A grant.

\section{Preliminaries and Notation}\label{sec:prelims}

We let $X$ and $E_\twist$ be as above.
We denote by $(\cdot, \cdot)_{\bundle}$ the Hermitian bundle metric on 
$\bundle$ that is induced from the sesquilinear inner product $(\cdot,\cdot)_V$ 
on $V$. We denote by $\ang{\cdot, \cdot}_{\bundle}$ the bilinear 
metric on~$\bundle$ corresponding to the bundle metric~$(\cdot, 
\cdot)_{\bundle}$.
We abbreviate the norm $\abs{v}_{\bundle} = \sqrt{(v,v)_{\bundle}}$ of any $v 
\in \bundle$ by $\abs{v}$.

By Selberg's Lemma \cite[Lemma~8]{Selberg_lemma}, there is a finite cover 
\[
\widetilde{X} = \tilde\group \bs \h
\]
of $X$ such that the Fuchsian group $\tilde{\Gamma}$ is a torsion-free subgroup of~$\group$.
We denote the pull-back of $\bundle$ under the covering map $\widetilde{X} \to 
X$ by~$\widetilde{E}$, which becomes a vector bundle over $\widetilde{X}$.
We call an operator $A$ acting on the sections of $E_\chi$ a 
\emph{pseudodifferential operator of order} $m\in\R$ if its 
pull-back,~$\widetilde{A}$, under the map $\widetilde{X} \to 
X$ is a pseudodifferential operator of order $m$, acting on the sections 
of~$\widetilde{E}$. 

In the case of the $1$-sphere $\Sph^1$, pseudodifferential operators have a very 
simple characterization using Fourier series, which we recall now. To that end 
let $A \colon  \CI(\Sph^1) \to \CI(\Sph^1)$ be a continuous linear operator.
As proven by McLean~\cite[Theorem 4.4]{McLean91},  $A$ is a 
pseudodifferential operator of order $m\in\R$ if and only if
\begin{align*}
    a(x,\xi) \coloneqq e^{-2\pi i\ang{x,\xi}} A\bigl( e^{2\pi i 
\ang{\cdot,\xi}} \bigr)\,, \quad (x,\xi) \in \Sph^1 \times \Z\,,
\end{align*}
is a periodic symbol of order $m$. This means that $a \in \CI(\Sph^1 \times 
\Z)$, and for all $b,c \in \N_0$, we have
\begin{align}\label{eq:periodic_symbol_estimates}
    \abs{ \pa_x^b \triangle_\xi^c a(x,\xi)} \lesssim_{b,c} \ang{\xi}^{m - c}\,.
\end{align}
Here, $\triangle_\xi$ denotes the discrete derivative, i.e., 
\[
\triangle_\xi a(x,\xi) \coloneqq u(x,\xi +1) - u(x,\xi)\,.
\]
Further, $\lesssim$ indicates an upper bound with implied constants.  More 
precisely, for any set~$Y$ and any functions~$a,b \colon  Y \to \R$,  
we write
\[
a\lesssim b\qquad\text{or}\qquad a(y) \lesssim b(y)
\]
if there exists a constant $C > 0$ such that for all $y \in Y$ we have 
\[
\abs{a(y)} \leq C \abs{b(y)}\,.
\]
If the constant, $C$, depends  on additional 
parameters, we indicate the dependence in the subscript.

Let $H$ be a Hilbert space and let $B\colon H \to H$ be a bounded operator. The 
non-zero eigenvalues of $(B^* B)^{1/2}$ are called the \emph{singular values} of 
the operator $B$. We denote these singular values by $\mu_k(B)$, 
$k\in\N$, listed in decreasing order.

For $z \in \C$ we define the \emph{Japanese bracket} $\ang{z} \coloneqq \left(1 + \abs{z}^2\right)^{1/2}$.

We use the convention to call a function, $f$, meromorphic \emph{on} an open 
set $U \subseteq \C$ if there exists a discrete subset, $P$, of~$U$ such that 
$f$, considered as a function, is defined on $U\setminus P$ only, and $f$ is 
holomorphic on~$U\setminus P$ and has poles (of finite order, which might be 
zero) at the points in~$P$.

\section{The Scattering Matrix for the Model Cylinders}\label{sec:scattering_matrix_for_model}

In this section we present the structure of the twisted scattering matrix for 
the model ends. We discuss the model funnel in Section \ref{sec:model-funnel} 
and the model cusp in Section \ref{sec:parab_cyl}.
The analysis was originally done in \cite[Section 4]{DFP}. Here we restrict to 
presenting the main results only.

\subsection{Model funnel}\label{sec:model-funnel}

Let $\ell \in (0,\infty)$ and set $\omega \coloneqq 2\pi/\ell$. We define the \emph{hyperbolic cylinder} as the quotient $C_\ell \coloneqq \ang{h_\ell} \bs \h$, where $h_\ell.z = e^\ell z$.
We may change coordinates via \[z = e^{\omega^{-1} \phi} \, \frac{e^r + i}{e^r 
- i}\] to $(r,\phi) \in \R \times \R / 2\pi\Z \cong C_\ell$ in such a way that 
the induced metric from the hyperbolic plane becomes
\begin{align*}
    g_{C_\ell}(r,\phi) \coloneqq dr^2 + \frac{\ell^2}{4\pi^2} \cosh^2 r\, d\phi^2\,.
\end{align*}
We define the \emph{model funnel} as 
\[
F_\ell \coloneqq \{ (r,\phi) \in C_\ell \colon r > 0\}
\]
with the metric $g_{F_\ell} \coloneqq g_{C_\ell}|_{F_\ell}$.
The canonical boundary defining function is $\rho_f(r,\phi) = \cosh(r)^{-1}$.

Taking the boundary defining function as a coordinate function, we may rewrite 
the funnel metric as 
\begin{align}\label{eq:funnelmetric_rho}
    g_{F_\ell}(\rho,\phi) = \rho^{-2} \left( \frac{\ell^2}{4\pi^2} d\phi^2 + \frac{d\rho^2}{1 - \rho^2} \right)\,.
\end{align}
The volume form is
\begin{align*}
    d\mu_{F_\ell} = \frac{\ell}{2\pi} \frac{d\rho \, d\phi}{\rho^2 \,\sqrt{1 - \rho^2}}\,.
\end{align*}
We also define the metric restricted to the boundary at infinity 
\begin{align*}
    g_{\pa_\infty F_\ell}(\phi,\pa_\phi) &\coloneqq \rho^2 g_{F_\ell}(\rho,\phi, 0,\pa_\phi)|_{\rho = 0}\\
    &= \frac{\ell^2}{4\pi^2} \,d\phi^2\,
\end{align*}
and denote the corresponding measure by
\begin{align*}
    d\sigma_{\pa_\infty F_\ell} \coloneqq \frac{\ell}{2\pi} \, d\phi\,.
\end{align*}
The Laplacian acting on functions $F_\ell \to \C$ takes the form
\begin{align}\label{eq:funnellaplace_rho}
    \Delta_{F_\ell} = - \rho^2 (1 - \rho^2) \pa_{\rho}^2 + \rho^3 \pa_\rho - \frac{4 \pi^2}{\ell^2} \rho^2 \pa_\phi^2\,.
\end{align}

Let $\twist \colon \ang{h_\ell} \to \Unit(V)$ be a finite-dimensional unitary representation. As above, we denote by~$\Delta_{C_\ell,\twist}$ the Laplacian acting on sections of the vector bundle $\bundle = \ang{h_\ell} \bs ( \h \times V)$ over $C_\ell$.
The Laplacian $\Delta_{F_\ell,\twist}$ is the restriction of the Laplacian $\Delta_{C_\ell,\twist}$ 
to $F_\ell$ with Dirichlet boundary conditions at $r = 0$. We will also denote 
by $E_\twist$ the restriction of $E_\twist$ to $F_\ell$. 
It was shown in \cite[Proposition 4.12]{DFP} that the resolvent of the model funnel, $(\Delta_{F_\ell,\twist}-s(1-s))^{-1}$, admits a meromorphic continuation to $\C$ as an operator
\begin{align*}
    R_{F_\ell,\twist}(s) \colon L^2_\cpt(F_\ell,\bundle) \to L^2_\loc(F_\ell,\bundle)\,.
\end{align*}
The multiset of its resonances, $\ResSet_{F_\ell,\twist}$,  is given by 
\begin{align}\label{eq:resonances_funnel}
	\ResSet_{F_\ell,\twist} \coloneqq \bigcup_{\lambda \in \EV(\twist(h_\ell))} \bigcup_{p \in \{\pm 1\}}
	\left( -(1+2\N_0) + p \ell^{-1} \left( \log\lambda + 2\pi i \Z\right)\right),
\end{align}
where $\EV(\twist(h_\ell))$ denotes the multiset of eigenvalues of 
$\twist(h_\ell)$. See \cite[Proposition 4.12]{DFP}.

Let $\psi \in \CI(\overline{F_\ell} \times \overline{F_\ell})$ 
such that $\psi$ is supported away from the diagonal and $s \in \C$ is not a 
pole of the resolvent $R_{F_\ell,\twist}$.
By \cite[Proposition 4.12]{DFP}, we have
\begin{align}\label{eq:asymptotics_funnel}
\psi R_{F_\ell,\twist}(s; \cdot, \cdot ) 
&\in (\rho_f \rho_f')^{s}\CI(\overline{F_\ell} \times \overline{F_\ell}, \bundle \boxtimes \bundle')\,,
\end{align}
where $\bundle \boxtimes \bundle'$ is the exterior tensor product of $\bundle$ 
and its dual $\bundle'$ defined by
\[
 \bigr(\bundle \boxtimes \bundle'\bigr)_{(x,\varphi)} \coloneqq 
 (\bundle)_x \otimes (\bundle')_\varphi\,.
\]
By \eqref{eq:asymptotics_funnel},
\begin{align}\label{eq:poisson-model-funnel}
E_{F_\ell,\twist}(s; r, \phi, \phi') \coloneqq \lim_{r' \to \infty} 
(\rho_f(r'))^{-s} R_{F_\ell,\twist}(s; r, \phi, r', \phi')
\end{align}
is well-defined. This allows us to introduce the \emph{Poisson operator}
\begin{align*}
E_{F_\ell,\twist}(s) &\colon \CI(\pa_\infty F_\ell, \bundle|_{\pa_\infty F_\ell}) \to \CI(F_\ell, \bundle),
\\
(E_{F_\ell,\twist}(s) f)(r,\phi) &\coloneqq \frac{\ell}{2\pi} \int_0^{2\pi} E_{F_\ell,\twist}(s;r,\phi,\phi') f(\phi') \, d\phi'.
\end{align*}

We now recall the Fourier expansion of the Poisson operator.
Let $(\psi_j)_{j=1}^{\dim V}$ be an eigenbasis of $\twist(h_\ell)$ with 
eigenvalues $\lambda_j = e^{2\pi i\vartheta_j}$, ${j = 1, \ldots, \dim V}$.
Let $\kappa \in \R$ and $s \in \C \setminus \left( -1 - 2 \N_0 \pm i \omega \kappa \right)$. We define
\begin{align}\label{def:beta}
\beta_\kappa(s) & \coloneqq \frac{1}{2} \gammafunc\left(\frac{s + i \omega \kappa + 1}{2}\right) \gammafunc\left(\frac{s - i \omega \kappa + 1}{2}\right).
\end{align}
We recall that the \emph{regularized hypergeometric function} $\FF(a,b;c;z)$ is defined for 
$a,b, c \in \C $ and $z\in\C$, $|z| < 1$, by the power 
series (see \cite[Theorem 9.1]{Olver74})
\begin{align*}
	\FF(a,b;c;z) \coloneqq  \sum_{n=0}^\infty 
	\frac{\gammafunc(a+n)\gammafunc(b+n)}{\gammafunc(a)\gammafunc(b)}\frac{1}{
		\gammafunc(c+n)} \cdot \frac{z^n}{n!}\,.
\end{align*}
For arbitrary $\kappa \in \R$, $s \in \C$ and $r \geq 0$, we define 
\begin{align*}
v_\kappa^0(s;r) & \coloneqq \tanh(r) (\cosh(r))^{-s}
\FF\left( \frac{s + i \omega \kappa + 1}{2}, \frac{s - i \omega \kappa + 1}{2}; \frac32; \tanh(r)^2 \right)\,.
\end{align*}
It was shown in \cite[Remark~4.13]{DFP} that
\begin{align*}
E_{F_\ell,\twist}(s;r,\phi,\phi') \psi_j
= \frac{1}{\ell}\sum_{k\in\Z} e^{i(k+\vartheta_j)(\phi-\phi')} E_{F_\ell,\twist}(s;r)_k^j \psi_j\,,
\end{align*}
where 
\begin{align*}
E_{F_\ell,\twist}(s;r)_k^j \coloneqq
\frac{\beta_{k + \vartheta_j}(s) v_{k + \vartheta_j}^0(s;r)}{\gammafunc(s+\frac{1}{2})}\,.
\end{align*}

By \cite[Lemma 6.15]{DFP},
we have for $\eps \in (0,1/2)$, $\varphi \in \CcI(F_\ell)$,
there exist $C, c > 0$ such that for all $s \in \C$ with $\Re s > \eps$
we have the estimate
\begin{align}\label{eq:svalues-poisson}
\mu_j(\varphi E_{F_\ell,\twist}(s)) \leq e^{C \ang{s} - c j}\,.
\end{align}
Moreover, the scattering matrix 
\[
S_{F_\ell,\twist}(s) \colon \CI(\pa_\infty F_\ell, \bundle|_{\pa_\infty F_\ell}) \to \CI(\pa_\infty F_\ell, \bundle|_{\pa_\infty F_\ell})
\]
was defined in \cite[(66)]{DFP} via the Fourier coefficients of its Schwartz kernel,
\begin{align*}
    S_{F_\ell,\twist}(s;\phi,\phi') \psi_j = \frac{1}{\ell} \sum_{k \in \Z} e^{i(k+\vartheta_j)(\phi - \phi')} S_{F_\ell,\twist}(s)_k^j \psi_j\,,
\end{align*}
where 
\begin{align}\label{eq:smatrix-funnel}
    S_{F_\ell,\twist}(s)_k^j \coloneqq \frac{\gammafunc(\frac12-s) \beta_{k+\vartheta_j}(s)}{ \gammafunc(s- \frac12) \beta_{k+\vartheta_j}(1-s)}\,.
\end{align}

From the Fourier expansion, we obtain that
\begin{align*}
    [E_{F_\ell,\twist}(1-s) S_{F_\ell,\twist}(s)]_k^j = - \frac{\beta_{k+\vartheta_j}(s) v^0_{k+\vartheta_j}(1-s)}{\gammafunc(s + \frac12)}
\end{align*}
and therefore
\begin{align}\label{eq:smatrix-funnel-intertwine}
    E_{F_\ell,\twist}(1-s) S_{F_\ell,\twist}(s) = -E_{F_\ell,\twist}(s)
\end{align}
since $v^0_\kappa(s) = - v^0_\kappa(1-s)$ by a connection formula (see \cite[p. 93]{Borthwick_book}).
We note that by \cite[(67)]{DFP}, 
\begin{align}\label{eq:poisson-asymptotics-funnel}
(2s-1) E_{F_\ell,\twist}(s;r) f \sim \sum_{m=0}^\infty \rho_f^{1-s + 2m} 
a_m(s) + \sum_{m=0}^\infty \rho_f^{s +2m} b_m(s)\,,
\end{align}
where the coefficient functions~$a_m, b_m$ for $m\in\N_0$ are meromorphic, with the leading coefficient functions being 
\[
a_0(s) = f\quad\text{and}\quad b_0(s) = S_{F_\ell,\twist}(s) f\,.
\]
From this, we obtain for $\Re s < 1/2$ that
\begin{align*}
    S_{F_\ell,\twist}(s) = (2s - 1) (\rho_f \rho_f')^{-s} R_{F_\ell,\twist}(s) \vert_{\pa_\infty F_\ell \times \pa_\infty F_\ell} \,.
\end{align*}

In what follows we will argue that the scattering matrix is a 
pseudodifferential operator on $\pa_\infty F_\ell$ and calculate its principal 
symbol.
From \cite[Eq. 5.11.13]{DLMF}, we have that for $a, b \in \C$ and $\abs{\arg(z)} 
< \pi - \eps$ for fixed $\eps > 0$,
\begin{align*}
    \frac{\gammafunc(z+a)}{\gammafunc(z+b)} \sim z^{a-b} \sum_{k=0}^\infty \frac{G_k(a,b)}{z^k}\,,
\end{align*}
for some $G_k(a,b) \in \C$.
This immediately implies that for $y  \to \infty$,
\begin{align*}
\frac{\gammafunc(iy+a)}{\gammafunc(iy+b)} \frac{\gammafunc(-iy+a)}{\gammafunc(-iy+b)} & \sim \left( (iy)^{a-b} \sum_{k=0}^\infty \frac{G_k(a,b)}{(iy)^k} \right) \left( (-iy)^{a-b} \sum_{k=0}^\infty \frac{G_k(a,b)}{(-iy)^k} \right) \\
&= \abs{y}^{2 (a - b)}\sum_{k=0}^\infty \frac{1}{y^k} \sum_{n=0}^{k} e^{\frac{\pi i}{2}   (k-2 n)} G_n(a,b) G_{k-n}(a,b).
\end{align*}
Taking $y = \omega \kappa/2$, $a = 1/2 + s/2$ and $b = 1- s/2$, we obtain
\begin{align}\label{gamma_expansion}
& \frac{\gammafunc\left(\frac{s + i \omega \kappa + 1}{2}\right) 
\gammafunc\left(\frac{s - i \omega \kappa + 
1}{2}\right)}{\gammafunc\left(\frac{2-s + i \omega \kappa}{2}\right) 
\gammafunc\left(\frac{2-s - i \omega \kappa}{2}\right)}  
\\ & 
\phantom{place}  \sim  \left|\frac{\omega \kappa}{2}\right|^{2 s-1} 
\sum_{k=0}^\infty \left(\frac{2}{\omega \kappa}\right)^k \sum_{n=0}^{k} 
e^{\frac{\pi i}{2}   (k-2 n)} G_n(a,b) G_{k-n}(a,b). \nonumber
\end{align}
We note that the terms in~\eqref{gamma_expansion} with odd $k$ vanish, since 
the left hand side is even as a function of $\omega \kappa$. The definition in
\eqref{def:beta} combined with~\eqref{gamma_expansion} shows that
\begin{align*}
    \frac{\beta_{\kappa}(s)}{\beta_\kappa(1-s)} &= 
    \frac{\gammafunc\left(\frac{s + i \omega \kappa + 1}{2}\right) \gammafunc\left(\frac{s - i \omega \kappa + 1}{2}\right)}{\gammafunc\left(\frac{2-s + i \omega \kappa}{2}\right) \gammafunc\left(\frac{2-s - i \omega \kappa}{2}\right)} \\
    &\sim \sum_{k=0}^\infty \left|\frac{\omega \kappa}{2}\right|^{2s-1 -2k} \sum_{n=0}^{2k} e^{\pi i   (k- n)} G_n(a,b) G_{2k-n}(a,b).
\end{align*}
By \cite[Eq. 5.11.15]{DLMF}, the leading coefficient is given by $G_0(a,b)^2 = 1$.
Combining this with~\eqref{eq:smatrix-funnel}, we obtain that
\begin{align*}
    S_{F_\ell,\twist}(s)^j_k \sim 2^{1-2s} 
\frac{\gammafunc(\frac12-s)}{\gammafunc(s-\frac12)} \abs{(k + \vartheta_j) 
\omega}^{2s-1}\,.
\end{align*}
The full asymptotic expansion now implies that 
$S_{F_\ell,\twist}(s)^j_k$ satisfies the symbol estimates for global 
pseudodifferential operators on the torus, 
as stated in~\eqref{eq:periodic_symbol_estimates}. Hence,
\begin{align}\label{eq:smatrix-funnel-psido}
S_{F_\ell,\twist}(s) \in \Psi^{2\Re s - 1}(\pa_\infty F_\ell, \bundle|_{\pa_\infty F_\ell}), \quad s 
\not \in \ResSet_{F_\ell,\twist} \cup \left(\N_0 + \frac12\right)\,.
\end{align}

We define the \emph{reduced scattering matrix} $\tilde{S}_{F_\ell,\twist}(s)$ as 
follows:
we consider the invertible elliptic pseudodifferential operator $\Lambda \in 
\Psi^1(\pa_\infty F_\ell,\bundle|_{\pa_\infty F_\ell})$ defined by 
\begin{align*}
    \Lambda \psi_j = \sum_{k \in \Z} e^{i (k+ \vartheta_j)(\phi - \phi')} 
\ang{k} \psi_j\,,
\end{align*}
set 
\[
G(s) \coloneqq \gammafunc\left(s+\frac12\right) \id_{\CI(\pa_\infty F_\ell, 
\bundle|_{\pa_\infty F_\ell})}
\]
and define 
\begin{align*}
    \tilde{S}_{F_\ell,\twist}(s) \coloneqq G(s) \Lambda(s) S_{F_\ell,\twist}(s) \Lambda(1-s)^{-1} G(1-s)^{-1}\,, 
\end{align*}
for  $s \not \in \ResSet_{F_\ell,\twist} \cup (\N_0 + 1/2)$.
A straightforward calculation shows that the Fourier coefficients of 
$\tilde{S}_{F_\ell,\twist}(s)$ are 
\begin{equation}\label{eq:Fell_spole}
\begin{aligned}
    \tilde{S}_{F_\ell,\twist}(s)^j_k &= \frac{\gammafunc(s+\frac12)}{\gammafunc(\frac12-s)} \ang{k}^{-2s + 1} S_{F_\ell,\twist}(s)^j_k \\
    &= \ang{k}^{-2s+1} \frac{ (s-\frac12) \beta_{k+\vartheta_j}(s)}{\beta_{k+\vartheta_j}(1-s)}\,.
\end{aligned}
\end{equation}
Since the right-hand side of the last equation is defined for all $s \not \in \ResSet_{F_\ell,\twist}$, the scattering matrix is defined as an operator
$\tilde{S}_{F_\ell,\twist}(s) \in \Psi^{0}(\pa_\infty F_\ell, 
\bundle|_{\pa_\infty F_\ell})$ for $s \not \in \ResSet_{F_\ell,\twist}$.
Taking advantage of this property, we can characterize the resonances in terms 
of the scattering matrix.

\begin{prop}\label{prop:Fell_spole_resonance}
    Let $s\in\C$, $\Re s < 1/2$ and let $m\in\N$. Then the reduced scattering matrix~$\tilde{S}_{F_\ell,\twist}$ has a pole of rank~$m$ at~$s$ if and only if $s$ is a resonance of multiplicity~$m$ of~$\Delta_{F_\ell,\twist}$. In this case, $m=m_{X,\twist}(s)$.
\end{prop}
\begin{proof}
    By definition of $\beta_\kappa$, we have that
    \begin{align*}
        \ang{k}^{-2s+1} \frac{ (s-\frac12)}{\beta_{k+\vartheta_j}(1-s)}
    \end{align*}
    is holomorphic and non-zero for $\Re s < 1/2$. Therefore, the poles counted with multiplicities of $\tilde{S}_{F_\ell,\twist}(s)^j_k$ are given by
    the multiset
    \begin{align*}
        \bigcup_{p \in \{\pm 1\}} \left( -(1 + 2\N_0) + 2\pi p \ell^{-1} (\vartheta_j + k) \right)\,.
    \end{align*}
    Hence, the poles of $\tilde{S}_{F_\ell,\twist}(s)$ are given by the multiset \eqref{eq:resonances_funnel}.
\end{proof}

Finally, we recall the singular value estimate for the scattering matrix from 
\cite[Lemma 6.14]{DFP}. 
For this, we will define functions $d_k \colon \C \to \C$ for $k \in \N$,
which have poles contained in the set of resonances of $\Delta_{F_\ell,\twist}$.
We set
\begin{align*}
\tilde{\mathcal{R}}_0 & \coloneqq 1 - 2\N_0\,,
\\
\mathcal{R}_0 & \coloneqq 1 - 2\N_0 + i\omega \Z\setminus\{0\}\,,
\\
\mathcal{R}_{1/2} & \coloneqq 1 - 2\N_0 + i \omega \left(\frac12 + \Z\right)\,,
\\
\mathcal{R}_{\vartheta} & \coloneqq \bigcup_{p \in \{\pm 1\}}
\left( 1 - 2\N_0 + ip\hspace{1pt}\omega (\vartheta + \Z)\right),\quad \vartheta 
\not \in \left\{0,\frac12\right\}\,,
\end{align*}
where we denote by $m_{\vartheta}$ the multiplicity of the eigenvalue
$\lambda = e^{2\pi i\vartheta}$ of $\twist(h_\ell)$.
We can assume without loss of generality that $\vartheta \in [0, 1)$.
Denote by $d_{\C}$ the Euclidean distance on $\C$.
For $k \in \N$ we define $d_{k,\vartheta}(s)$ as follows: for $\vartheta 
\in (0,1) \setminus \{1/2\}$, we set
\begin{align*}
d_{k,\vartheta}(s) \coloneqq 
\begin{cases}
\dist(s,\mathcal{R}_{\vartheta})^{-1}, & k \leq m_{\vartheta},
\\
1, & k > m_{\vartheta}
\end{cases}
\end{align*}
and for $\vartheta \in \{0,1/2\}$, we set
\begin{align*}
d_{k,\vartheta}(s) \coloneqq 
\begin{cases}
\dist(s,\mathcal{R}_{\vartheta})^{-1}, & k \leq 2m_{\vartheta},
\\
1, & k > 2m_{\vartheta}.
\end{cases}
\end{align*}
Moreover, we define the function $\tilde{d}_{k,0}$ by
\begin{align*}
\tilde{d}_{k,0}(s) \coloneqq 
\begin{cases}
\dist(s,\tilde{\mathcal{R}}_{0})^{-2}, & k \leq m_0,
\\
1, & k > m_0.
\end{cases}
\end{align*}
Finally, we set
\begin{align}\label{eq:def_dks}
d_k(s) \coloneqq \tilde{d}_{k,0}(s) \cdot \prod_{\vartheta} d_{k,\vartheta}(s)\,.
\end{align}
It is shown in \cite[Lemma~6.14]{DFP} that
for any $\eps \in (0,1/2)$ there exists $C > 0$ such that for
$s \in \C$ with $\Re s < 1/2 - \eps$,
\begin{align}\label{eq:svalues-smatrix}
\mu_k(S_{F_f,\twist}(s)) \leq e^{C\ang{s}} \ang{s}^{1 - 2\Re s} \times  
\begin{cases}
d_k(s), & k \leq \max \{m_0, 2m_{\vartheta_j}\}\,,
\\
k^{2\Re s - 1}, & k > \max \{m_0, 2m_{\vartheta_j}\}\,.
\end{cases}
\end{align}

\subsection{Parabolic cylinders}\label{sec:parab_cyl}
We now turn to the parabolic cylinder, where the structure of the resolvent is 
slightly simpler than for the hyperbolic cylinder.

The \emph{parabolic cylinder} is given by $C_\infty \coloneqq \ang{T} \bs \h$, where $T.z \coloneqq 
z+1$. We can choose as fundamental domain the set
\begin{align*}
\funddom \coloneqq \{x+iy \in \h \colon x \in (0,1)\}\,.
\end{align*}
With the coordinates $(\rho,\phi) = (y^{-1}, (2\pi)^{-1} x)$, the 
induced Riemannian metric reads
\begin{align}\label{eq:cuspmetric_rho}
    g_{C_\infty} = \frac{d\rho^2}{\rho^2} + \rho^2 \frac{d\phi^2}{4\pi^2}
\end{align}
and $\rho_c(\rho,\phi) = \rho$, where $\rho_c$ is the canonical boundary defining function.
In the $(x,y)$-coordinates the Laplacian is given by
\begin{align*}
    \Delta_{C_\infty} = -y^2 (\pa_x^2 + \pa_y^2)\,.
\end{align*}

Let $\twist \colon \ang{T} \to \Unit(V)$ be a finite-dimensional unitary representation.
We denote by $E_1(\twist(T))$ the eigenspace of $\twist(T)$ for 
eigenvalue~$1$, and we set $n_c^\twist \coloneqq \dim E_1(\twist(T))$.

The meromorphically continued resolvent $R_{C_\infty,\twist}(s)$ defines a 
continuous map
\begin{align}\label{eq:asymptotics_cusp}
\psi R_{C_\infty,\twist}(s) \colon \CcI(C_\infty,\bundle) \to 
\rho_c^{s-1} \CI(\overline{C_\infty}, \bundle)
\end{align}
provided that $s \not = 1/2$, where $\psi$ is any element 
of~$\CI(\overline{C_\infty})$ that is supported away from $\{y = 0\}$.
The only pole of $R_{C_\infty,\twist}(s)$ is at the point $s=1/2$ and its multiplicity is equal to $n_c^\twist$.

The integral kernel of the resolvent $R_{C_\infty,\twist}(s)$ admits a Fourier decomposition.
For any $j \in \{1, \dotsc, \dim V\}$, the Fourier decomposition of the non-vanishing matrix coefficients $R_{C_\infty,\twist}(s; z, z')^j$ is  given by
\begin{equation}\label{eq:fdecompparabcyl}
R_{C_\infty,\twist}(s; z, z')^j = \sum_{k \in \Z} e^{2 \pi i (k + \vartheta_j) (x-x')} u_{2\pi(k + \vartheta_j)}(s;y, y ') \,.
\end{equation}
Here, the maps $u_\kappa$ for $\kappa \in \R$ are defined as follows: for 
$\kappa \in \R \setminus\{0\}$, we set
	\begin{align*}
	u_\kappa(s;y,y') \coloneqq \begin{cases}
	\sqrt{yy'} I_{s-1/2}(\abs{\kappa}y) K_{s-1/2}(\abs{\kappa}y'), & y\leq 
	y',\\
	\sqrt{yy'} K_{s-1/2}(\abs{\kappa}y) I_{s-1/2}(\abs{\kappa}y'), & y > y',
	\end{cases}
	\end{align*}
	where $I_{s-1/2}$ and $K_{s-1/2}$ is the modified Bessel function of the 
first and the second kind, respectively (see \cite[\S~3.7]{Watson44}). 
Moreover, for $\kappa=0$ and $s \not = 1/2$ we set
	\begin{align*}
	u_0(s;y,y') \coloneqq \frac{1}{2s-1}
	\begin{cases}
	y^s (y')^{1-s}, & y \leq y',
	\\
	y^{1-s} (y')^{s}, & y > y'.
	\end{cases}
	\end{align*}
	
The \emph{Poisson operator} $E_{C_\infty,\twist}(s)$ is given by
\begin{align*}
    E_{C_\infty,\twist}(s) &: \CI(\pa_c C_\infty, \bundle) \to \CI(C_\infty, \bundle)\\
    (E_{C_\infty,\twist}(s) u)(x,y) &\coloneqq \frac{y^s}{2s-1} u(x)\,,
\end{align*}
where $u \in \CI(\pa_c C_\infty) \cong \C^{n_c^\twist}$.
The Schwartz kernel of $E_{C_\infty,\twist}(s)$ is given by
\begin{equation}\label{eq:poisson-model-cusp}
    \begin{aligned}
        E_{C_\infty,\twist}(s;x,y,x') &= \frac{y^s}{2s-1} \id_{E_1(\twist(T))} \\
        &= \lim_{y' \to \infty} \rho_c(y')^{1-s} R_{C_\infty,\twist}(s;x,y,x',y')
    \end{aligned}
\end{equation}
by the Fourier decomposition in~\eqref{eq:fdecompparabcyl}.
In particular, 
\[
E_{C_\infty,\twist}(s;x,y,x') = E_{C_\infty,\twist}(s;y)
\]
is independent of $x,x'$.

\section{Analysis of the Resolvent}\label{sec:resolvent}

In this section, we discuss fine-structure properties of the resolvent of 
$\LapTwist$. 
We start, in Theorem~\ref{thm:resolvent}, with a decomposition of its resolvent 
into interior and residual terms, which are then discussed separately in more 
detail. In Section~\ref{sec:resolvent_at_resonance}, we give a description of the resolvent near a 
resonance. 
In Section~\ref{sec:no-resonances}, we prove that on the line $\Re(s) = 1/2$ there are no resonances except for potentially $s=1/2$.
Moreover, in Proposition \ref{prop:no-eigenvalues}, we prove that if the 
hyperbolic surface~$X$ has infinite volume, then $\LapTwist$ has no eigenvalues larger than $1/4$.

As in \cite[Section 3.2.3]{DFP}, we take advantage of the decomposition 
\[
X = K \sqcup X_f \sqcup X_c\,,
\]
where $K$ is compact and $X_f$ and~$X_c$ are finite collections of funnels and cusps, respectively.
For $\bullet \in \{f,c\}$ and $r \in [0,\infty)$, we choose a cutoff function $\eta_{\bullet,r} \in \CI(X)$ such that
\begin{align*}
\eta_{\bullet,r}(x) = 
\begin{cases} 
1, & \text{if $d(X \setminus X_\bullet,x) < r$}\,,
\\ 
0, & \text{if $d(X\setminus X_\bullet,x) > r+\frac12$}\,.
\end{cases}
\end{align*}
We fix $s_0 \in \C$ with sufficiently large real part
(such that $s(1-s)$ is sufficiently far away from the spectrum of 
$\Delta_{X,\twist}$) 
and denote by $n_f$ and $n_c$ the number of connected components of $X_f$ and $X_c$, respectively.
As in \cite[Section~5]{DFP}, we set
\begin{align}
M_i &\coloneqq \eta_{f,2}\eta_{c,2} \ResTwist(s_0) \eta_{f,1}\eta_{c,1} \,, \\
M_f(s) &\coloneqq (1 - \eta_{f,0}) R_{X_f,\twist}(s) ( 1 - \eta_{f,1})\,,\label{eq:resolvent-mf} \\
M_c(s) &\coloneqq (1 - \eta_{c,0}) R_{X_c,\twist}(s) ( 1 - \eta_{c,1})\,,\label{eq:resolvent-mc}
\end{align}
where 
\begin{equation*}
\begin{gathered}
R_{X_f,\twist}(s)\colon L^2(X_f,\bundle) \to L^2(X_f,\bundle),
\\
R_{X_f,\twist}(s) \coloneqq  R_{X_{f,1},\twist_1}(s) \oplus \dotsc \oplus R_{X_{f,{n_f}},\twist_{n_f}}(s)
\end{gathered}
\end{equation*}
and 
\begin{equation*}
\begin{gathered}
R_{X_c,\twist}(s)\colon L^2(X_c,\bundle) \to L^2(X_c,\bundle),
\\
R_{X_c,\twist}(s) \coloneqq R_{X_{c,1},\twist_1}(s) \oplus \dotsc \oplus R_{X_{c,{n_c}},\twist_{n_c}}(s)\,.
\end{gathered}
\end{equation*}
Further, we set 
\[
M(s) \coloneqq M_i  + M_f(s) + M_c(s)
\]
and, as in 
\cite[(85)]{DFP}, we define 
\begin{align}\label{eq:parametrix-error}
L(s) \coloneqq L_i(s) + L_f(s) + L_c(s)\,,
\end{align}
where 
\begin{align}
L_i(s) & \coloneqq -[\LapTwist, \eta_{f,2}\eta_{c,2}] \ResTwist(s_0) \eta_{f,1}\eta_{c,1}
\nonumber
\\
&\phantom{\coloneqq -} + (s(1-s)- s_0(1-s_0)) M_i(s_0)\,, 
\nonumber
\\
L_f(s) &\coloneqq [\LapTwist, \eta_{f,0}] R_{X_f,\twist}(s) (1 - \eta_{f,1})\,,\label{eq:resolvent-lf}
\intertext{and}
L_c(s) &\coloneqq [\LapTwist, \eta_{c,0}] R_{X_c,\twist}(s) (1 - \eta_{c,1})\,.\label{eq:resolvent-lc}
\end{align}
It follows that
\begin{align}\label{eq:parametrix-applied}
(\LapTwist - s(1-s)) M(s)= \id - L(s)\,.
\end{align}
It was proven in \cite[Section 5]{DFP} that $(\id - L(s))^{-1}$ exists as a 
meromorphic family in $s \in \C$.

\begin{thm}\label{thm:resolvent}
	Let $X = \group \bs \h$ be geometrically finite and let $\twist \colon \group \to \Unit(V)$ be  a finite-dimensional unitary representation of $\group$.
    For $s\in \C$ not a pole of neither~$\ResTwist(s) $ nor $M_f(s)$ nor 
$M_c(s)$, the resolvent admits a decomposition
    \begin{align*}
    \ResTwist(s) = \tilde{M}_i(s) + M_f(s) + M_c(s) + Q(s)\,,
    \end{align*}
    where
    \begin{itemize}
    \item $\tilde{M}_i(s)$ is a compactly supported pseudodifferential operator 
of order~$-2$,
    \item $M_f(s)$ and $M_c(s)$ are as in \eqref{eq:resolvent-mf} and \eqref{eq:resolvent-mc}, respectively, and
    \item $Q(s)$ is an integral operator with the integral kernel $Q(s;\cdot, \cdot)$ satisfying \[ Q(s;\cdot, \cdot) \in (\rho_f\rho_f')^s (\rho_c\rho_c')^{s-1}\CI(\overline{X} \times \overline{X}, \bundle \boxtimes \bundle').\]
    \end{itemize}
\end{thm}
\begin{remark}
    The product $\overline{X} \times \overline{X}$ is not a smooth manifold 
(even in the absence of orbifold points).
    The reason is that the geodesic boundary at infinity of a cusp end is a single point. Blowing up each parabolic fixed point to a $1$-sphere, we obtain a orbifold with boundary $\overline{X}'$.
    We define smooth functions on $\overline{X} \times \overline{X}$ as the set of function that pullback to smooth functions on $\overline{X}' \times \overline{X}'$.

    Note that blowing up a parabolic fixed point amounts to introducing coordinates $(\rho, \phi)$ as in Section~\ref{sec:parab_cyl}, where $\{\rho = 0\} \cong \Sph^1$ is the blowup of the parabolic fixed point.
\end{remark}

\begin{proof}[Proof of Theorem~\ref{thm:resolvent}]
We set 
\begin{align}\label{k:definition}
K(s) \coloneqq (\id - L(s))^{-1} L(s)\,.
\end{align}
Note that 
\begin{align}\label{eq:Kss0}
    \id + K(s) = (\id - L(s))^{-1}\,.
\end{align}
Then \eqref{eq:parametrix-applied} implies
\begin{align*}
\ResTwist(s) = M(s)(\id + K(s))\,.
\end{align*}
For notational simplicity, define $\eta_3 \coloneqq  \eta_{f,3}\eta_{c,3}$.
We now split $M(s) K(s)$ as
\begin{align*}
M(s) K(s) = \eta_3 M(s) K(s) \eta_3 + Q(s)\,,
\end{align*}
where
\begin{align*}
Q(s) \coloneqq (1-\eta_3) M(s) K(s) \eta_3 + M(s) K(s) (1-\eta_3)\,.
\end{align*}
Moreover, we define $\tilde{M}_i(s) \coloneqq M_i + \eta_3 M(s) K(s) \eta_3$ and note that
\begin{align*}
\ResTwist(s) = \tilde{M}_i(s) + M_f(s) + M_c(s) + Q(s)\,.
\end{align*}
We now have to show that $\tilde{M}_i(s)$ and $Q(s)$ have the claimed properties.

\subsection*{Interior term}
The operator $M_i$ is a compactly supported pseudodifferential operator by definition, so it suffices to show that
$\eta_3 M(s) K(s) \eta_3$ is a pseudodifferential operator of order at most $-2$.
By~\eqref{eq:parametrix-error} we have 
\begin{align*}
\eta_3 L(s) = L(s)\,.
\end{align*}
Now equation~\eqref{eq:Kss0} directly implies that
\begin{align*}
    K(s)\eta_3 = (\id - L(s))^{-1} \eta_3 - \eta_3\,.
\end{align*}
From $\eta_3 L(s) = L(s)$, we obtain
\begin{align*}
    (\id + K(s) \eta_3) (\id - L(s) \eta_3) = \id\,.
\end{align*}
Consequently,
\begin{align}\label{eq:K-eta3}
    \id + K(s) \eta_3 = (\id - L(s) \eta_3)^{-1}\,.
\end{align}
for $s$ close to $s_0$.
By the identity theorem for holomorphic functions,
the equality in~\eqref{eq:K-eta3} is valid for all $s \in \C$.
Formally, we can also obtain this from the geometric series,
\begin{align*}
    \id + K(s)\eta_3 = \id + \sum_{k > 0} L(s)^k \eta_3 = (\id - L(s)\eta_3)^{-1}\,.
\end{align*}
We have that
\begin{align*}
    L(s) \eta_3 = \left( s(1-s) - s_0(1-s_0) \right) M_i + \tilde{Q}(s)\,,
\end{align*}
where $\tilde{Q}(s)$ is compactly supported and smoothing.
Thus, $L(s) \eta_3$ is a pseudodifferential operator of order $-2$ and
therefore
\begin{align*}
    K(s)\eta_3 = (\id - L(s)\eta_3)^{-1} L(s) \eta_3
\end{align*}
is also a pseudodifferential operator of order $-2$.
By the definition of $M(s)$, the operator $\eta_3 M(s)$ is a
pseudodifferential operator of order $-2$
and hence
$\eta_3 M(s) K(s) \eta_3$ is a pseudodifferential operator of order $-4$.

\subsection*{Residual term}
To study the operator $Q(s)$, we start by considering the operator $M(s) K(s) (1-\eta_3)$. We use \eqref{k:definition} to show that 
\begin{align*}
    K(s) = L(s) (\id + K(s))\,.
\end{align*}
Since $L(s)$ maps to compactly supported smooth sections, we use the explicit 
calculations for the model resolvents to obtain that,  for any $\varphi \in 
\CcI(X, \bundle)$, we have
\begin{align*}
    L(s)^T \varphi \in (\rho_f)^s (\rho_c)^{s-1} \CI(\overline{X}, \bundle)\,.
\end{align*}
Moreover, the property 
\[
L(s) \varphi \in (\rho_f)^s (\rho_c)^{s-1} \CI(\overline{X}, \bundle)
\]
implies that
the integral kernel $K(s) (1 -\nobreak \eta_3)(\cdot,\cdot)$ of the operator 
$K(s) (1 - \eta_3)$ satisfies 
\begin{align*}
K(s) (1 - \eta_3)(\cdot,\cdot) \in (\rho_f \rho_c)^\infty (\rho_f')^s (\rho_c')^{s-1}\CI(\overline{X} \times \overline{X}, \bundle \boxtimes \bundle')
\end{align*}
and is compactly supported in the left-most variable.
Using that $M_i$ is a compactly supported pseudodifferential operator and $M_f(s)$ and $M_c(s)$ are given by the model resolvents, we conclude that the integral kernel of the operator $M(s) K(s) (1 - \eta_3)$ satisfies 
\begin{align*}
M(s) K(s) (1 - \eta_3)(\cdot,\cdot) \in (\rho_f\rho_f')^s (\rho_c\rho_c')^{s-1}
\CI(\overline{X} \times \overline{X}, \bundle \boxtimes_0 \bundle')\,.
\end{align*}
For $(1 - \eta_3) M(s) K(s)$, we use that $K(s) \eta_3$ is compactly supported 
and that the integral kernel of the operator $(1-\eta_3) M(s) K(s) \eta_3$  
satisfies
\begin{align*}
(1-\eta_3) M(s) K(s) \eta_3 (\cdot,\cdot) \in \rho_f^s \rho_c^{s-1} (\rho_f' \rho_c')^\infty
\CI(\overline{X} \times \overline{X}, \bundle \boxtimes \bundle')\,.
\end{align*}
This proves the theorem.
\end{proof}

We will now provide formula for $Q(s)$ restricted to the boundary that will be 
useful later on.
Let $\varphi \in \CI(\overline{X}, \bundle)$ such that $\eta_3 \varphi = 0$.
In this case $Q(s) \varphi$ simplifies to
\begin{align*}
Q(s) \varphi &= M(s) K(s) \varphi
\\
&= M(s) (\id - L(s))^{-1} L(s) \varphi\,.
\end{align*}
Using that $L(s) = \eta_3 L(s)$, we obtain
\begin{align*}
L(s) &= \eta_3 (\id - L(s) \eta_3) (\id - L(s) \eta_3)^{-1} L(s)
\\
&= (\id - L(s)) \eta_3 (\id - L(s) \eta_3)^{-1} L(s)\,.
\end{align*}
Hence, we have
\begin{align}\label{eq:Qs_product}
Q(s) \varphi &= M(s) \eta_3 (\id - L(s) \eta_3)^{-1} L(s) \varphi\,.
\end{align}

\subsection{Resolvent at a Resonance}\label{sec:resolvent_at_resonance}

Let $s_0 \in \C$ be a resonance of~$\LapTwist$. As in \cite[Section 6]{DFP}, we 
define the \emph{multiplicity} of the resonance $s_0$ as the number
\begin{align*}
m_{X,\twist}(s_0) \coloneqq \rank \int_{\gamma_{\eps,s_0}} \ResTwist(s) \, ds\,,
\end{align*}
where $\eps > 0$ is chosen such that the path $\gamma_{\eps,s_0}\colon [0,1] 
\to \C$ with 
\[
\gamma_{\eps,s_0}(t) \coloneqq s_0 + \eps e^{2\pi i t}
\]
encloses exactly one resonance (namely $s_0$).
We denote the multiset of resonances by
\begin{align*}
\ResSet_{X,\twist} \coloneqq \{ (s_0, m) \in \C \times \N \colon \text{$s_0$ is
a resonance, $m = m_{X,\twist}(s_0)$}\}
\end{align*}
and the multiset of resonances of the model funnel ends by
\begin{align}\label{eq:resset-ends}
    \ResSet_{X_f,\twist} \coloneqq \bigcup_{j=1}^{n_f} \ResSet_{X_{f,j},\twist_j}\,,
\end{align}
where the multiset $\ResSet_{X_{f,j},\twist_j}$ is given as in~\eqref{eq:resonances_funnel}.

In a small neighborhood of the resonance~$s_0$, the resolvent admits an 
expansion
\begin{align}\label{eq:resolvent-at-resonance}
\ResTwist(s) = \sum_{j=1}^p \frac{A_j(s_0)}{\left(s(1-s) - s_0(1-s_0) \right)^j} + H(s, s_0)\,
\end{align}
for some $p \in \N$, further referred to as the \emph{order} of the resonance, 
where, for $j=1,\ldots, p$, the coefficient~$A_j(s_0)$ is a finite 
rank operator, and the map~$s \mapsto H(s,s_0)$ is holomorphic in a small 
neighborhood of $s_0$.

Now let $s_0 \not = 1/2$ and fix $j=1,\ldots, p$.
We multiply \eqref{eq:resolvent-at-resonance} 
by \[\left( s(1-s) - s_0(1-s_0) \right)^{j-1}\] and integrate both sides along the path ${\gamma_{\eps,s_0}}$.
We substitute $\lambda = s(1-s)$ and use $d\lambda = (1-2s) ds$. The path of the integration changes to 
\[
\tilde{\gamma}_{\eps, s_0}(t) = s_0(1-s_0) + (1-2s_0) \eps e^{2 \pi i t}  + \eps^2 e^{4 \pi i t}\,.
\]
For $s_0\neq 1/2$ and $\eps$ small enough, $\tilde{\gamma}_{\eps, s_0}(t)$ winds around $s_0(1-s_0)$ once. Applying the Cauchy integration formula, we get 
\begin{align}\label{def:Ajs0}
	A_j(s_0) = \frac{1}{2\pi i} \int_{\gamma_{\eps,s_0}} (1 - 2s)  \left( s(1-s) - s_0(1-s_0) \right)^{j-1} \ResTwist(s) \, ds
\end{align}
for any $j=1, \ldots, p$. We note that for $j=1$, the equality \eqref{def:Ajs0} was obtained in the proof of~\cite[Lemma 2.4]{GuZw97}.

Note that \eqref{def:Ajs0} implies that the operator $A_j(s_0)$ is symmetric. Together with  $m_{X,\twist}(s_0) = \rank A_1(s_0)$, this yields
\begin{align}\label{eq:resolvent-A_1}
A_1(s_0) = \sum_{\ell,m=1}^{m_{X,\twist}(s_0)} a_1^{\ell,m}(s_0) \phi_\ell \ang{\phi_m, \cdot}\,,
\end{align}
where $a_1(s_0) = (a_1^{\ell,m}(s_0))_{\ell,m=1}^{m_{X,\twist}(s_0)}$ is a symmetric invertible matrix and \[\phi_j \in \rho^{-N}L^2(X,\bundle)\] for $j=1, \ldots, m_{X,\twist}(s_0)$ and $\Re(s_0) > 1/2 - N$ for any $N \in \N$.
The definition of the resolvent implies that for any $j = 1, \dotsc, p-1$,
\begin{align*}
A_{j+1}(s_0) &= A_j(s_0) (\LapTwist - s_0(1-s_0)) \\ &= (\LapTwist - s_0(1-s_0)) A_j(s_0)\,,
\\
A_{p+1}(s_0) &= 0\,.
\end{align*}
Therefore, we obtain that
\begin{align}\label{eq:resolvent-A_k}
A_k(s_0) = \sum_{\ell,m=1}^{m_{X,\twist}(s_0)} a_k^{\ell,m}(s_0) \phi_\ell \ang{\phi_m, \cdot}\,,
\end{align}
where $a_k(s_0) \coloneqq a_1(s_0) d(s_0)^{k-1}$ for $d(s_0) \coloneqq a_1(s_0)^{-1} a_2(s_0)$. Note that the matrix $d(s_0)$ is nilpotent.

\subsection{Absence of Poles with \texorpdfstring{$\Re s = 1/2$}{Re s = 1/2}}\label{sec:no-resonances}
In this section, we will show that for $s \in \CC$ with $\Re s = 1/2$ there is at most one resonance at ${s=1/2}$. This will imply that there are no eigenvalues larger than $1/4$.

The Carleman estimate~\cite[Theorem~(7)]{Mazzeo91} reads in our setting as 
follows (cf. Borthwick~\cite[Lemma~7.6]{Borthwick_book}).

\begin{prop}\label{prop:carleman-est}
Let $F_\ell \subset C_\ell = \ang{h_\ell} \bs \h$ be a hyperbolic funnel and let $\twist\colon \ang{h_\ell} \to \Unit(V)$ be a unitary finite-dimensional representation. Denote by $\rho_f$ the boundary defining function of $\pa_\infty F_\ell$. Let $r_0, k \ge 0$ and suppose that $u \in \CI(X, \bundle)$ satisfies $u = O(\rho_f^\infty)$ and is supported in $\{r \geq r_0\}$, where $r$ denotes the distance to the geodesic boundary. For $r_0$ and $k$ sufficiently large there exists $C > 0$ independent of $k$ such that
\begin{align*}
    k^3 \int_{F_\ell} e^{2kr} \abs{u}^2 d\mu_{F_\ell} + k \int_{F_\ell} e^{2kr} \abs{\nabla_\twist u}^2 \, d\mu_{F_\ell} \leq C \int_{F_\ell} e^{2kr} \abs{ \Delta_{F_\ell,\twist} u}^2 \, d\mu_{F_\ell}\,.
\end{align*}
\end{prop}

The Carleman estimate implies the following result on unique continuations
(see~\cite[Proposition~7.4]{Borthwick_book} for the untwisted case).

\begin{prop}\label{prop:unique-cont}
Let $X = \group \bs \h$ be an infinite-volume hyperbolic surface and 
$\twist\colon \group \to \Unit(V)$ be a unitary finite-dimensional 
representation.
Suppose that $u \in \CI(X, \bundle)$ is a solution of $(\LapTwist - s(1-s)) u = 0$ for some $s \not \in -\N_0/2$. If
\begin{align}\label{prop:4.3:equ}
u|_{X_{f,j}} \in \rho_f^{s+1} \CI(\overline{X_{f,j}}, \bundle)
\end{align}
for some $j=1, \ldots, n_f$, then $u \equiv 0$.
\end{prop}
We adapt the proof of \cite[Proposition~7.4]{Borthwick_book} to the twisted case.
\begin{proof}[Proof of Proposition~\ref{prop:unique-cont}]
    Without loss of generality, we assume that $X$ has only one funnel end, that is $n_f = 1$ and  $X_f = X_{f,1}$.
    We prove the proposition in two steps.
    
    \paragraph{\textbf{Step 1:}} We want to show by induction  that
    \[u|_{X_f} \in \rho_f^{s+n} \CI(\overline{X_f}, \bundle), \quad \forall n \in \N.\] 
    The base case is true by \eqref{prop:4.3:equ}.  Suppose that $u|_{X_f} \in 
\rho_f^{s+n} \CI(\overline{X_f}, \bundle)$ for some $n \in \N$.
    Write $u|_{X_f} = \rho_f^{s+n} v$, where $v \in \CI(\overline{X_f}, \bundle)$.
    Using \eqref{eq:funnellaplace_rho} we obtain
    \begin{align*}
        (\LapTwist - s(1-s)) \rho_f^{s+n} v = n(1 - 2s - n) \rho_f^{s+n} v + O(\rho_f^{s+n+1})\,.
    \end{align*}
    Since $u$ solves $(\LapTwist - s(1-s)) u = 0$, it follows that $v = 
O(\rho_f)$ under the assumption that $s\not \in -\N_0/2$. 
    Therefore,
    \begin{align*}
        u|_{X_f} \in \rho_f^{s+n+1} \CI(\overline{X_f}, \bundle)\,.
    \end{align*}
    By induction, we obtain that $u|_{X_f} \in \rho_f^\infty \CI(\overline{X_f}, \bundle)$.
    \paragraph{\textbf{Step 2:} We want to show that $u \equiv 0$.}
    Choose $r_0, r_1 \in (0,1)$ with $r_1 > r_0$  and choose $\eta \in \CI([0,1])$ such that $\eta(r) = 1$ for $r \leq r_0$ and $\eta(r) = 0$ for $r \geq r_1$.
    The function $ \eta(\rho_f) u|_{X_f}$  satisfies the assumptions of 
Proposition~\ref{prop:carleman-est}. We hence obtain 
    \begin{align*}
        k^3 \int_{X_f} \rho_f^{-2k} \abs{\eta(\rho_f)}^2 \abs{u}^2 \,d\mu_X \leq C \int_{X_f} \rho_f^{-2k} \abs{ \LapTwist \eta(\rho_f) u}^2 \, d\mu_X
    \end{align*}
    for $k > 0$ large enough. 
        Denote 
    \[
    I_1 \coloneqq (1 + \abs{s(1-s)}^2) \int_{X_f \cap \{\rho_f \leq r_0\}}\rho_f^{-2k}\abs{u}^2 \,d\mu_X\,.
    \]
        Using the equation $(\LapTwist - s(1-s)) u = 0$ and the fact that 
$\eta(r) = 1$ for $r \leq r_0$, we obtain
    \begin{align*}
         \frac{I_1 \cdot  k^3 }{1+|s(1-s)|^2} &=  k^3 \int_{X_f \cap \{\rho_f \leq r_0\}} \rho_f^{-2k} \abs{u}^2 \,d\mu_X \\
        &\leq k^3 \int_{X_f} \rho_f^{-2k} \abs{\eta(\rho_f)}^2 \abs{u}^2 \,d\mu_X \\
        &\leq C \int_{X_f} \rho_f^{-2k} \abs{ \LapTwist \eta(\rho_f) u}^2 \, d\mu_X \\
        &= C    \int_{X_f \cap \{r_0 \leq \rho_f \leq r_1\}}   \rho_f^{-2k} \abs{ \LapTwist \eta(\rho_f) u}^2 \, d\mu_X \\
        &\quad +  C \int_{X_f \cap \{\rho_f \leq r_0\}}  \rho_f^{-2k} \abs{ 
\LapTwist \eta(\rho_f) u}^2 \, d\mu_X \\
        &\leq C \left( I_2 + I_3+ I_1 \right)\,,
    \end{align*}
    where 
    $C = C(r_0, r_1, \eta) > 0$ and
    \begin{align*}
        I_2 &\coloneqq (1 + \abs{s(1-s)}^2) \int_{X_f \cap \{r_0 \leq \rho_f 
\leq r_1\}}\rho_f^{-2k}\abs{u}^2 \,d\mu_X\,,\\
        I_3 &\coloneqq \int_{X_f \cap \{r_0 \leq \rho_f \leq r_1\}} 
\rho_f^{-2k} \abs{\nabla_{X,\twist} u}^2 \,d\mu_X\,.
    \end{align*}
    Setting $C' = (1 + \abs{s (1-s)}^2)^{-1} C^{-1}$, we  rewrite the above estimate as
    \begin{align*}
        I_1 \leq (C' k^3 - 1)^{-1} ( I_2 + I_3 ) \,.
    \end{align*}
    We  estimate $I_2$ and $I_3$ by
    \begin{align*}
        I_2 + I_3 &\leq C'' \int_{r_0}^{r_1} \rho^{-2k-1} \,d\rho \\
        &= \frac{ C''}{2k r_0^{2k}} \left(1 - \left(\frac{r_1}{r_0}\right)^{-2k} \right) \,,
    \end{align*}
    for some $C'' > 0$, which depends on $r_0, r_1, s$, and $u$, but is independent of $k$.
    Therefore we arrive at
    \begin{align*}
        \int_{X_f \cap \{\rho_f \leq r_0\}}\abs{u}^2 \,d\mu_X &\leq \frac{r_0^{2k}}{1 + \abs{s (1-s)}^2} I_1 \\
        &\leq \frac{C'' \left(1 - \left(\frac{r_1}{r_0}\right)^{-2k} \right)}{2k(1 + \abs{s (1-s)}^2)(C' k^3 - 1)}  \,.
    \end{align*}
    Letting $k \to \infty$, we obtain that $\norm{u}_{L^2(X_f \cap \{\rho_f \leq r_0\}, \bundle)} = 0$ and consequently $u = 0$ on $X_f \cap \{\rho_f \leq r_0\}$.
    By standard uniqueness results of elliptic differential operators, we conclude that $u = 0$ everywhere.
\end{proof}

In the case $\Re s = 1/2$ and $s \not = 1/2$, we can prove a better result following \cite[Lemma~7.7]{Borthwick_book}.

\begin{prop}\label{prop:unique-cont-1/2}
Let $X$ and $\twist$ be as above, and let $\Re s = 1/2$ with $s \not = 1/2$. If $u \in \CI(X, \bundle)$ satisfies $u|_{X_{f,j}} \in \rho_f^s \CI(X_{f,j}, \bundle)$ for some $j\in\{1, \ldots, n_f\}$ and
\begin{align*}
(\LapTwist - s(1-s)) u = 0\,,
\end{align*}
then $u \equiv 0$.
\end{prop}
\begin{proof}
    Without loss of generality, we may suppose that $n_f=1$ and $X_f = X_{f,1}$.
    We take local coordinates $(r, \phi) \in \R_+ \times \R / 2\pi \Z \cong X_f$.
    We have that $u(r, \phi + 2\pi) = (\twist(h_\ell) u)(r,\phi)$, where $h_\ell \in \group$ is the unique (up to inversion)
    hyperbolic element associated to the funnel end $X_f$ and $\ell \in (0, \infty)$ is the length of the central geodesic of $X_f$ (see \cite[Section 3.2.3]{DFP} for details).
   
    Let $\eps > 0$ and let $\psi \in \CI(\R_+)$ be real-valued with $\psi(t) = 0$ for $t \leq 1$ and $\psi(t) = 1$ for $t \geq 2$.
    Set $\psi_\eps \in \CI(\overline{X})$ with $\psi_\eps(\rho,\theta) = \psi(\rho/\eps)$ for $(\rho,\theta) \in X_f$ and $\psi_\eps = 1$ on $X \setminus X_f$.
    Since $\Re s = 1/2$, we have that $s(1-s) \in \R$ and thus
    \begin{align*}
        0 &= \int_X \left( \overline{s(1-s)} (\psi_\eps u, u)_{\bundle} - s(1-s) (u, \psi_\eps u)_{\bundle} \right) \,d\mu_X \\
        &= \int_X ( [\LapTwist, \psi_\eps \cdot \id_{\bundle}] u, u)_{\bundle} \, d\mu_X \\
        &= \int_{X_f} ( [\LapTwist, \psi_\eps \cdot \id_{\bundle}] u, u)_{\bundle} \, d\mu_X \,.
    \end{align*}
    The function $u$ can be written as $u = \rho^s v$, where $v \in \CI(\overline{X_f}, \bundle)$.
    By assumption, we have that $\Re s = 1/2$, therefore $\abs{u}^2 = \rho \abs{v}^2$.
    Writing $x = \rho/\eps$, we obtain, using \eqref{eq:funnellaplace_rho}, 
that
    \begin{align*}
        ([\LapTwist, \psi_\eps \cdot \id_{\bundle}] u, u)_{\bundle} = - \eps x^2 (x \psi''(x) + 2s \psi'(x) )\abs{v(0, \phi)}^2 + O(\eps^2)\,.
    \end{align*}
    as $\eps \to 0$.
    It follows from \eqref{eq:funnelmetric_rho}, that the measure $d\mu_X$ restricted to $X_f$ is given by
    \begin{align*}
        d\mu_X|_{X_f} &= \rho^{-2} \frac{\ell}{2\pi} \frac{d\rho\,d\phi}{\sqrt{1 - \rho^2}} \\
        &= \eps^{-1} x^{-2} \frac{\ell}{2\pi} \frac{dx\,d\phi}{\sqrt{1 - \eps^2 x^2}}\,.
    \end{align*}
    Therefore we have, as $\eps\to 0$, that
    \begin{align*}
        \int_{X_f} ( [\LapTwist, \psi_\eps & \cdot \id_{\bundle}] u, 
u)_{\bundle} \, d\mu_X 
        \\ & = -\frac{\ell}{2\pi} \int_1^2 \int_0^{2\pi} (x \psi''(x) + 2s 
\psi'(x) )\abs{v(0, \phi)}^2 \,dx \,d\phi + O(\eps)\,.
    \end{align*}
    We calculate that $\int_1^2 \psi'(x) = 1$ and $\int_1^2 x \psi''(x) \,dx 
= -1$ and therefore
    \begin{align*}
        \int_{X_f} ( [\LapTwist, \psi_\eps \cdot \id_{\bundle}] u, u)_{\bundle} \, d\mu_X &= (1 - 2s) \frac{\ell}{2\pi} \int_0^{2\pi} \abs{v(0, \phi)}^2 \,d\phi + O(\eps)
    \end{align*}
    as $\eps \to 0$.
    This implies that $u|_{X_f} \in \rho_f^{s+1} \CI(\overline{X_f}, \bundle)$. 
    Together with Proposition~\ref{prop:unique-cont} this implies the claim.
\end{proof}

Proposition~\ref{prop:unique-cont-1/2} implies
almost no resonances of the critial line.

\begin{cor}\label{cor:no-poles}
For $\Re s = 1/2$ and $s \not = 1/2$, the resolvent $\ResTwist$ has no pole at 
$s$.
\end{cor}
\begin{proof}
    By \eqref{eq:resolvent-at-resonance}, we have that
    \begin{align*}
        \ResTwist(s) = \sum_{j=1}^p \frac{A_j(s_0)}{\left(s(1-s) - s_0(1-s_0)\right)^j} + H(s, s_0)\,,
    \end{align*}
    where $p\in \N$ is the order of the resonance, $A_j(s_0)$, $j = 1, \ldots, p$ are finite rank operators and $H(s,s_0)$ is holomorphic in $s$ near $s = s_0$.
    Let $\psi \in \CcI(X,\bundle)$ and write $u = A_p(s_0) \psi$. 
    By the definition of the resolvent, we have that
    \begin{align*}
        (\LapTwist - s_0(1-s_0)) u = 0\,.
    \end{align*}
    By Theorem~\ref{thm:resolvent}, we have that $u \in \rho_f^{s_0} \rho_c^{s_0-1} \CI(\overline{X},\bundle)$.
    For $\Re s_0 = 1/2$ and $s \not = 1/2$, Proposition~\ref{prop:unique-cont-1/2} implies that $u = 0$
    and consequently $A_p = 0$. This shows that $\ResTwist(s)$ is holomorphic near $s_0$.
\end{proof}

\begin{prop}\label{prop:no-eigenvalues}
    The Laplacian $\LapTwist$ has no eigenvalues in the interval $[1/4, \infty)$. 
\end{prop}
\begin{proof}
    Let $\lambda \in [1/4, \infty)$ and set $s \coloneqq 1/2 + i\sqrt{\lambda - 1/4}$.
    This implies that $\lambda = s(1-s)$.
    Assume that $\lambda$ is an eigenvalue of $\LapTwist$, then there
    exists a function $u \in L^2(X,\bundle)$ such that
    \begin{align*}
        (\LapTwist - s(1-s)) u = 0\,.
    \end{align*}
    Since $X$ has infinite volume, there is at least one funnel end, which we 
will denote by $X_f$. We choose coordinates $(r,\phi) \in X_f$ as in 
Section~\ref{sec:model-funnel}.
    Choose $\psi \in \CcI(X_f,\bundle\vert_{X_f})$ such that $\supp \psi \subset \{r \geq 2\}$. Then we have by \eqref{eq:parametrix-applied} that 
    \begin{align*}
        (\LapTwist - s(1-s)) M_f(s) \psi &= \psi - L_f(s) \psi\,.
    \end{align*}
    Let $\eps > 0$. Integrating by parts,
    we have that
    \begin{align*}
        \eps (2s-1+\eps) &\int_{X_f} \ang{M_f(s+\eps) \psi, u}_{\bundle} \,d\mu_X \\
        &= \int_{X_f} \ang{(\LapTwist - (s+\eps)(1-s-\eps)) M_f(s+\eps) \psi, u}_{\bundle} \,d\mu_X \\
        &= \int_{X_f} \ang{\psi - L_f(s + \eps)\psi, u}_{\bundle} \,d\mu_X \,.
    \end{align*}
    By the Cauchy--Schwarz inequality, we have that
    \begin{align*}
        \left| \int_{X_f} \ang{M_f(s+\eps) \psi, u}_{\bundle} \,d\mu_X \right| \leq \norm{ M_f(s+\eps) \psi } \norm{u}
    \end{align*}
    and using \eqref{eq:asymptotics_funnel} and \eqref{eq:resolvent-mf}, we obtain that \[\rho^{-s-\eps} M_f(s+\eps) \psi \in \CI(\overline{X_f},\bundle)\,.\]
    Therefore, we can estimate
    \begin{align*}
        \norm{ M_f(s+\eps) \psi } \leq \sup_{X_f} \abs{ \rho^{-s-\eps} M_f(s+\eps) \psi } \norm{\rho^{s+\eps}}\,,
    \end{align*}
    where the first factor in the right-hand side is bounded by a constant and the second factor is $O(\eps^{-1/2})$ by a direct calculation.
    This implies that
    \begin{align*}
        \int_{X_f} \ang{\psi - L_f(s + \eps)\psi, u}_{\bundle} \,d\mu_X = O(\eps^{1/2}) \quad\text{ as } \eps \to 0\,.
    \end{align*}
    By the fundamental lemma of calculus of variations, this implies that
    \begin{align*}
        u(z) = (L_f(s)^T u)(z)
    \end{align*}
    for $z \in X_f \cap \{r \geq 2\}$.
    By the definition of $L_f(s)$, \eqref{eq:resolvent-lf}, we have that ${u|_{X_f} \in \rho_f^s \CI(\overline{X_f},\bundle)}$.
    Set $u_0(s) \coloneqq \rho_f^{-s} u|_{\pa_\infty X_f}$.
    Since ${u \in L^2(X, \bundle)}$ and $\Re s = 1/2$, it follows that $u_0(s) \equiv 0$ and therefore \[u|_{X_f} \in \rho_f^{s+1} \CI(\overline{X_f},\bundle).\]
    Proposition~\ref{prop:unique-cont} now finishes the proof.
\end{proof}

\section{Scattering Determinant}\label{sec:scattering-determinant}
In this section, we prove Theorems~\ref{thm:scattering-determinant} and \ref{thm:smatrix-onehalf}. 
We start with introducing the Poisson operator and studying its properties in Section~\ref{sec:Poisson_operator}.
In Section~\ref{sec:scattering_matrix}, we define the scattering matrix and show the correspondence of resonances and poles of the scattering matrix for $\Re s < 1$ and $s \neq 1/2$.
In Section~\ref{sec:scat_matrix_12}, we study the behavior of $\ResTwist(s)$ near $s=1/2$ and prove Theorem~\ref{thm:smatrix-onehalf}.
In Section~\ref{sec:scattering_poles}, we recall the basics of the Gohberg-Sigal theory and obtain a relation of scattering poles and resonances for $\Re(s) \le 1$.
In Section~\ref{sec:relative_scattering_matrix}, we introduce the relative scattering matrix and the relative scattering determinant and, finally, prove Theorem~\ref{thm:scattering-determinant}.

\subsection{Poisson Operator}\label{sec:Poisson_operator}

Before we define the scattering matrix, we introduce the Poisson operator,
which maps sections $\CI(\pa_\infty X, \bundle|_{\pa_\infty X})$ to solutions of the equation $(\LapTwist - s(1-s))u = 0$ with prescribed asymptotics at the boundary at infinity. 
The construction is similar to the one in the untwisted
case~\cite[(2.23)-(2.25)]{GuZw97}, but in our case
the Poisson operator acts on sections of vector bundles and we have be more careful due to the compactification in the cusp, which depends on the representation $\twist$.

We recall that the ideal boundary at infinity $\pa_\infty X$ is a disjoint 
union of circles (representing funnel ends) and points (representing cusp ends) 
and that
we have the decomposition
\begin{align}\label{eq:ideal_boundary_decomposition}
\pa_\infty X = \pa_f X \sqcup \pa_c X\,.
\end{align}
For $j\in\{1,\ldots, n_f\}$ and $s \not \in \ResSet_{X,\twist}$ we define the 
map 
\[
E_{X,\twist}^{f,j}(s)\colon \CI(\pa_\infty X_{f,j}, \bundle|_{\pa_\infty X_{f,j}}) \to \CI(X, \bundle)
\]
by its Schwartz kernel
\begin{align*}
E_{X,\twist}^{f,j}(s,z,\theta') \coloneqq \left. (\rho_f')^{-s} \ResTwist(s;z, z') \right|_{X \times \pa_\infty X_{f,j}} \,.
\end{align*}
The restriction is well-defined by Theorem~\ref{thm:resolvent} and~\eqref{eq:asymptotics_funnel}.
Similarly, for $j\in\{1,\ldots, n_c\}$ and $s \not \in \ResSet_{X,\twist}$ we 
define
\begin{align*}
E_{X,\twist}^{c,j}(s) &: \CI(\pa_\infty X_{c,j}, \bundle|_{\pa_\infty X_{c,j}}) \to \CI(X, \bundle)\,,
\\
E_{X,\twist}^{c,j}(s,z,\theta') &\coloneqq \left. (\rho_c')^{1-s} \ResTwist(s;z, z') \right|_{X \times \pa_\infty X_{c,j}} \,.
\end{align*}
The restriction is well-defined by Theorem~\ref{thm:resolvent} and 
\eqref{eq:asymptotics_cusp}. Further, by \eqref{eq:Qs_product}, 
\eqref{eq:resolvent-lc}, and \eqref{eq:resolvent-mc}, the map 
$E_{X,\twist}^{c,j}(s,z,\theta')$ is independent of $\theta'$ and
defines an operator $\CI(\pa_\infty X_{c,j}, \bundle|_{\pa_\infty X_{c,j}}) \to \CI(X, \bundle)$.
We denote this two-variable function by $E_{X,\twist}^{c,j}$ as well.
We obtain the \emph{Poisson operator} defined by its Schwartz kernel as follows:
\begin{align*}
E_{X,\twist}(s) &: \CI(\pa_\infty X,\bundle|_{\pa_\infty X}) \to \CI(X,\bundle)\,,\\
(E_{X,\twist}(s)\psi)(z) &\coloneqq \sum_{j=1}^{n_f} \frac{\ell_j}{2\pi} \int_0^{2\pi} E_{X,\twist}^{f,j}(s,z,\theta') f_j(\theta') d\theta' 
+ \sum_{j=1}^{n_c} E_{X,\twist}^{c,j}(s,z) a_j\,,
\end{align*}
where $\psi = (f_1,\dotsc,f_{n_f}, a_1, \dotsc, a_{n_c}) \in \CI(\pa_\infty 
X,\bundle|_{\pa_\infty X})$.
The transposed operator
\[
E_{X,\twist}(s)^T \colon \CcI(X, \bundle) \to \CI(\pa_\infty X, \bundle|_{\pa_\infty X})
\]
is given by
\[
E_{X,\twist}(s)^T u = (f_1,\dotsc,f_{n_f}, a_1, \dotsc, a_{n_c})\,,
\]
where
\begin{align*}
f_j(\theta) &= \int_X E_{X,\twist}^{f,j}(s,z',\theta) u(z') d\mu_X(z')\,,
\\
a_j &= \int_X E_{X,\twist}^{c,j}(s,z') u(z') d\mu_X(z')\,.
\end{align*}
By the same arguments as in the proof of \cite[Lemma 4.14]{DFP},
we can express the difference of resolvents in terms of the Poisson operator for general hyperbolic surfaces.

\begin{prop}\label{prop:resolvent-difference}
    Let $X = \group \bs \h$ be a geometrically finite hyperbolic surface and $\twist \colon \group \to U(V)$ a finite-dimensional unitary representation. For $s \not \in \ResSet_{X,\twist} \cup (1-\ResSet_{X,\twist})$, we have
    \begin{align*}
    \ResTwist(s) - \ResTwist(1-s) = (1 - 2s) E_{X,\twist}(s) E_{X,\twist}(1-s)^T\,.
    \end{align*}
\end{prop}
\begin{proof}
    We follow the proof of \cite[Lemma 4.14]{DFP}, but we have to take care of the multiple ends.

    We fix a fundamental domain $\funddom \subset \h$ of $X$. Then the 
bundle $\bundle \boxtimes \bundle'$ is trivial and can be identified with 
$\End(V)$.
    We fix $z,w \in \funddom$. We define the coefficients of $\ResTwist(s;z,w)$ as
    \begin{align*}
        R_{jk}(s;z,w) \coloneqq \ang{ \ResTwist(s;z,w)e_j, e_k}_V\,.
    \end{align*}
    We also set
    \begin{align*}
        R_{jk}^T(s;z,w) \coloneqq \ang{ \ResTwist^T(s;z,w)e_j, e_k}_V\,,
    \end{align*}
    where $\ResTwist^T(s;z,w)$ denotes the Schwartz kernel of the operator $\ResTwist(s)^T$.

    We calculate
    \begin{align*}
        R_{jk}(s;z,w) &- R_{jk}(1-s;z,w) \\
        &= \lim_{\eps \to 0} \int_{\rho(z') > \eps} \sum_{m=1}^{\dim V} \bigg( R_{jm}(s;z,z') \LapTwist R_{mk}(1-s;z',w) \\
        &\phantom{=}\quad - \LapTwist R_{jm}(s;z,z') R_{mk}(1-s;z',w) \bigg) 
d\mu_X(z') \\
        &= \lim_{\eps \to 0} \int_{\rho(z') > \eps} \sum_{m=1}^{\dim V} \bigg( R_{jm}(s;z,z') \LapTwist R_{km}^T(1-s;w,z') \\
        &\phantom{=}\quad - \LapTwist R_{jm}(s;z,z') R_{km}^T(1-s;w,z') \bigg) 
d\mu_X(z') \\
        &= \lim_{\eps \to 0} \int_{\rho(z') = \eps} \sum_{m=1}^{\dim V} \bigg( - R_{jm}(s;z,z') \pa_{\nu} R_{km}^T(1-s;w,z') \\
        &\phantom{=}\quad + \pa_{\nu} R_{jm}(s;z,z') R_{km}^T(1-s;w,z') \bigg) 
d\sigma_{X_\eps}(z') \,.
    \end{align*}
    Here, $X_\eps \coloneqq \{z \in X \colon \rho(z) = \eps\}$, and 
$d\sigma_{X_\eps}$ is the induced measure on $X_\eps$.
    If we pick $\eps > 0$ sufficiently small, then the area of integration 
splits into a disjoint union of funnel and cusp ends.
    Without loss of generality, we suppose that $X_f = X_{f,j}$ and we set 
$X_{f,\eps} \coloneqq X_f \cap X_\eps$.
    From \eqref{eq:funnelmetric_rho}, we see that $\pa_\nu = \rho \pa_\rho + O(\rho^2)$.
    For $z \in X$ with $\rho(z) > \eps$ and $z' \in X_{f,\eps}$, we have that
    \begin{align*}
        \ResTwist(s;z,z') = (\rho')^s E_{X,\twist}(s;z,\phi') + O( (\rho')^{s+1})
    \end{align*}
    and
    \begin{align*}
        \pa_\nu \ResTwist(s;z,z') &= -\rho' \pa_{\rho'} \ResTwist(s;z,z') \\
        &= -s (\rho')^s E_{X,\twist}(s;z,\phi') + O( (\rho')^{s+1})\,.
    \end{align*}
    Consequently, 
    \begin{align*}
       & - R_{j m}(s;z,z')  \pa_{\nu} R_{km}^T(1-s;w,z') \\ & \phantom{- 
R_{jm}(s;z,z')} = (1 - s) \eps E_{jm}(s;z,\phi') E_{km}^T(1-s;w,\phi') + 
O(\eps^2)\,
\intertext{and}
       & \pa_{\nu}  R_{jm}(s;z,z') R_{km}^T(1-s;w,z') \\ & \phantom{- R_{jm}(s;z,z')} = -s \eps E_{jm}(s;z,\phi') E_{km}^T(1-s;w,\phi') + O(\eps^2)\,.
    \end{align*}
    Moreover, $d\sigma_{X_{\eps}}|_{X_{f,\eps}} = (2\pi \eps)^{-1} \ell\, d\phi = \eps^{-1} d\sigma_{\pa_\infty X_f}$,
    where $\ell \in (0,\infty)$ is the length of the central geodesic associated to $X_f$.
    Therefore, we obtain that
    \begin{align*}
        &\int_{X_{f,\eps}} \left( - R_{jm}(s;z,z') \pa_{\nu} R_{km}^T(1-s;w,z') \right. \\ 
        &  \phantom{\int_{X_{f,\eps}}  - R_{lm}(s;z,z')} \left. + \pa_{\nu} R_{jm}(s;z,z') R_{km}^T(1-s;w,z') \right) d\sigma_{X_\eps}(z') \\
        &= (1-2s) \int_{\pa_\infty X_f} \left( E_{jm}(s;z,\phi') E_{km}^T(1-s;w,\phi') + O(\eps) \right) d\sigma_{\pa_\infty X_f}(\phi')\,.
    \end{align*}
    Letting $\eps \to 0$ proves the claim for the funnel ends.

    For the cusp ends, we also suppose without loss of generality that $n_c=1$ 
    and $X_c = X_{c,1}$ is a single cusp end.
    We set $X_{c,\eps} \coloneqq X_c \cap X_\eps$.
    We take coordinates $(\rho,\phi) \in \R_+ \times \R /2 \pi \Z \cong X_c$ as in Section~\ref{sec:parab_cyl} and calculate
    $g_{X_c}(\pa_\rho,\pa_\rho) = \rho^{-2}$ and therefore $\pa_\nu = -\rho' \pa_{\rho'}$.
    By the definition of the Poisson operator,
    we have for $z \in X$ with $\rho(z) > \eps$ and $z' \in X_{c,\eps}$ and as $\eps\to 0$ (hence $\rho'\to 0$), 
    \begin{align*}
        \ResTwist(s;z,z') = (\rho')^{s-1} E_{X,\twist}(s;z,\phi') + O( (\rho')^{s})
    \end{align*}
    and
    \begin{align*}
        \pa_\nu \ResTwist(s;z,z') &= -\rho' \pa_{\rho'} \ResTwist(s;z,z') \\
        &= -s (\rho')^{s-1} E_{X,\twist}(s;z,\phi') + O( (\rho')^{s})\,.
    \end{align*}
    Therefore,
    \begin{align*}
        &R_{jm}(s;z,z') \pa_{\nu} R_{km}^T(1-s;w,z') \\ & \phantom{- R_{jm}(s;z,z') }  = -(\rho')^{s-1} E_{jm}(s;z,\phi) (1-s) (\rho')^{-s} E_{km}^T(1-s;w,\phi') \\
        & \phantom{- R_{jm}(s;z,z') = } + O( (\rho')^{0})
    \end{align*}
    and
    \begin{align*}
        & \pa_{\nu} R_{jm}(s;z,z') R_{km}^T(1-s;w,z') \\  & \phantom{- R_{jm}(s;z,z') }  = -s (\rho')^{s-1} E_{jm}(s;z,\phi) (\rho')^{-s} E_{km}^T(1-s;w,\phi') + O( (\rho')^{0})\,.
    \end{align*}
    By \eqref{eq:cuspmetric_rho}, we have that $d\sigma_{X_{\eps}}|_{X_{c,\eps}} = \eps \frac{d\phi}{2\pi}$.
    Using that $E_{X,\twist}(s;z,\phi')$ is independent of $\phi'$, we arrive at
    \begin{align*}
        &\int_{X_{c,\eps}} \left( - R_{jm}(s;z,z') \pa_{\nu} R_{km}^T(1-s;w,z') \right.\\ 
        &  \phantom{\int_{X_{c,\eps}}  - R_{lm}(s;z,z')} \left. + \,\pa_{\nu} R_{jm}(s;z,z') R_{km}^T(1-s;w,z') \right) d\sigma_{X_\eps}(z') \\
        &= (1-2s) \frac{1}{2\pi} \int_0^{2\pi} E_{jm}(s;z) E_{km}^T(1-s;w) \,d\phi' + O(\eps) \\
        &= (1-2s) E_{jm}(s;z) E_{km}^T(1-s;w) + O(\eps)
    \end{align*}
    Taking $\eps \to 0$ yields the result.
\end{proof}

The Poisson operator $E_{X,\twist}(s)$ provides generalized eigenfunctions in the following sense.

\begin{prop}\label{prop:poisson-asymptotics}
Let $s \not \in \ResSet_{X,\twist}$. For any $\psi \in \CI(\pa_\infty X, \bundle|_{\pa_\infty X})$, we have
\begin{align}\label{eq:LapTwists1-sECchi}
(\LapTwist - s(1-s)) E_{X,\twist}(s) \psi = 0
\end{align}
and
\begin{align*}
E_{X,\twist}(s) \psi \in \rho_f^{1-s} \rho_c^{-s} \CI(\overline{X}, \bundle) + \rho_f^s \rho_c^{s-1} \CI(\overline{X}, \bundle)\,.
\end{align*}
If $s \not \in \Z / 2$, then we have the asymptotics
\begin{align}\label{eq:poisson-asymptotics}
(2s-1) E_{X,\twist}(s) \psi \sim \rho_f^{1-s} \rho_c^{-s} \psi + \rho_f^s \rho_c^{s-1} \phi_s\,,
\end{align}
where
$\phi_s \in \CI(\pa_\infty X, \bundle|_{\pa_\infty X})$ depends meromorphically on $s \in \C$.
\end{prop}

\begin{remark}
    In the case of the model funnel, this result follows directly from~\eqref{eq:poisson-asymptotics-funnel}.
\end{remark}

\begin{proof}[Proof of Proposition \ref{prop:poisson-asymptotics}]
It is straightforward to see that $E_{X,\twist}(s) \psi$ solves the equation~\eqref{eq:LapTwists1-sECchi}.
To obtain \eqref{eq:poisson-asymptotics}, we use the result on the structure of 
the resolvent, Theorem~\ref{thm:resolvent}. We have that
\begin{align*}
E_{X,\twist}^{f,j}(s;z,\theta') = \lim_{\rho' \to 0} (\rho')^{-s}\left( M_f(s;z,\rho',\theta') + Q(s;z,\rho',\theta') \right)
\end{align*}
and by \eqref{eq:resolvent-mf},
\begin{align*}
\lim_{\rho' \to 0} (\rho')^{-s} M_f(s;z,\rho',\theta') = (1 - \eta_{f,0}) E_{X_f,\twist}(s;z,\theta')\,,
\end{align*}
where $E_{X_f,\twist}(s)$ is defined by \eqref{eq:poisson-model-funnel}.
From the asymptotics of $Q(s)$, 
Theorem~\ref{thm:resolvent},
we obtain
\begin{align*}
E_{X,\twist}^{f,j} f_j(z) - (1 - \eta_{f,0}) E_{X_f,\twist}(s) f_j \in \rho_f^s \rho_c^{s-1} \CI(\overline{X},\bundle)\,.
\end{align*}

For the cusp ends, we have to be more careful, because the compactification at 
the cusp of the bundle $\bundle$ depends on the multiplicity of the 
eigenvalue~$1$ of $\twist(\gamma_j)$,
where $\gamma_j \in \group$ is a representative of the conjugacy class $[\gamma_j]$, associated to the cusp $X_{c,j}$.
Similar to the funnel case, we have
\begin{align*}
E_{X,\twist}^{c,j}(s;z,\theta') = \lim_{\rho' \to 0} (\rho')^{-s}\left( M_c(s;z,\rho',\theta') + Q(s;z,\rho',\theta') \right)\,.
\end{align*}
Using the notation of Section~\ref{sec:parab_cyl}, we have
\begin{align*}
\lim_{\rho' \to 0} (\rho')^{-s} M_c(s;z,\rho',\theta') =
(1 - \eta_{c,0}) \frac{\rho^{-s}}{2s-1} \id_{E_1(\twist(\gamma_j))}\,,
\end{align*}
where $E_1(\twist(\gamma_j))$ is the $1$-eigenspace of $\twist(\gamma_j)$.
Let \[\varphi \coloneqq \eta_3 (\id - L(s)\eta_3)^{-1} [\LapTwist, \eta_{c,0}] 
\frac{\rho^{-s}}{2s-1} \id_{E_1(\twist(\gamma_j))}\,.
\]
By \eqref{eq:Qs_product} we have
\begin{align*}
    \lim_{\rho' \to 0} (\rho')^{1-s} Q(s;z,\rho',\theta') = (M(s) \varphi)(z)\,.
\end{align*}
Therefore, we obtain that
\[
\lim_{\rho' \to 0} (\rho')^{1-s} Q(s;z,\rho',\theta') \in \rho^s_f \rho_c^{s-1} \CI(\overline{X},\bundle)\,.
\]
By the definition of the compactification of $\bundle$ at the cusp, we have for $a_j \in \CC^{n^\twist_{c,j}}$ that
\begin{align*}
E_{X,\twist}^{c,j} a_j - (1 - \eta_{c,0}) \frac{1}{2s-1} \rho^{-s} a_j \in \rho_f^s \rho_c^{s-1}\CI(\overline{X}, \bundle)\,.
\end{align*}
\end{proof}

\subsection{Scattering Matrix}\label{sec:scattering_matrix}
The scattering matrix intertwines the asymptotics of solutions of the equation $(\LapTwist - s(1-s)) u = 0$ as described in Proposition~\ref{prop:poisson-asymptotics}.
\begin{defi}\label{def:smatrix}
    For $s \not \in \ResSet_{X,\twist} \cup \Z / 2$, the \emph{scattering matrix} is
    given by
    \begin{align*}
    S_{X,\twist}(s) &\colon  \CI(\pa_\infty X, \bundle|_{\pa_\infty X}) \to \CI(\pa_\infty X, \bundle|_{\pa_\infty X})\,, 
    \\
    S_{X,\twist}(s) &\colon  \psi \mapsto \phi_s\,,
    \end{align*}
    where $\phi_s$ is defined by  \eqref{eq:poisson-asymptotics}. 
\end{defi}

We observe that
\begin{align}\label{property_of_scattering_under_adjoint_and_transpose}
S_{X,\twist}(s)^* = S_{X,\twist}(\bar{s})\,,\qquad
S_{X,\twist}(s)^T = S_{X,\twist}(s)\,,
\end{align}
where $S_{X,\twist}(s)^*$ is the adjoint of $S_{X,\twist}(s)$ with respect to the complex inner product on $L^2(\pa_\infty X, \bundle|_{\pa_\infty X})$
and $S_{X,\twist}(s)^T$ is the transposed operator.

\begin{prop}\label{prop:intertwine}
For any $s\in\C$, $s \not \in \ResSet_{X,\twist} \cup (1-\ResSet_{X,\twist})$ 
and any element $\psi \in \CI(\pa_\infty X, \bundle|_{\pa_\infty X})$, we have 
\begin{align}
E_{X,\twist}(1-s) S_{X,\twist}(s) \psi &= - E_{X,\twist}(s) \psi\,, \nonumber
\\
S_{X,\twist}(1-s) S_{X,\twist}(s) \psi &= \psi\,.
\label{eq:scattering-matrix-intertwine}
\end{align}
\end{prop}

\begin{proof}
It suffices to prove the statement for $\Re s \le 1/2$, $s \neq 1/2$ and $s \not 
\in \ResSet_{X,\twist} \cup (1 - \ResSet_{X,\twist})$. By 
Proposition~\ref{prop:resolvent-difference},
\begin{align*}
\ResTwist(s) - \ResTwist(1-s) = (1-2s) E_{X,\twist}(s) E_{X,\twist}(1-s)^T\,.
\end{align*}
Multiplying this equation from the left with $\rho_f^{-s} \rho_c^{1-s}$ and 
restricting to the boundary yields
\begin{align*}
E_{X,\twist}(s)^T - 0 = -S_{X,\twist}(s) E_{X,\twist}(1-s)^T\,.
\end{align*}
Using that $S_{X,\twist}(s)^T = S_{X,\twist}(s)$ we obtain the first claim.
In order to obtain \eqref{eq:scattering-matrix-intertwine}, we calculate
\begin{align*}
    E_{X,\twist}(s) S_{X,\twist}(1-s) S_{X,\twist}(s) \psi &= -E_{X,\twist}(1-s) S_{X,\twist}(s) \psi \\
    &= E_{X,\twist}(s) \psi\,.
\end{align*}
By \eqref{eq:poisson-asymptotics}, $E_{X,\twist}(s)$ is injective. This 
proves the claim.
\end{proof}

Proposition~\ref{prop:intertwine} together with 
Proposition~\ref{prop:resolvent-difference} implies that
\begin{align}\label{eq:resolvent-difference-smatrix}
\ResTwist(s) - \ResTwist(1-s) = (1-2s) E_{X,\twist}(1-s) S_{X,\twist}(s) E_{X,\twist}(1-s)^T\,.
\end{align}
It is convenient to use the identification
\begin{align*}
\CI(\pa_c X, \bundle|_{\pa_c X}) \cong \C^{n_c^\twist}\,,
\end{align*}
where $n_c^\twist = \sum_{j=1}^{n_c} n_{c,j}^\twist$.
Using the decomposition into funnel and cusp ends, we can write the scattering matrix as
\begin{align*}
S_{X,\twist}(s) = 
\begin{pmatrix} 
S_{X,\twist}^{\ff}(s) & S_{X,\twist}^{\fc}
\\
S_{X,\twist}^{\cf}(s) & S_{X,\twist}^{\cc}
\end{pmatrix}\,,
\end{align*}
where
\begin{gather*}
S_{X,\twist}^{\ff}(s) \colon \CI(\pa_f X, \bundle|_{\pa_f X}) \to \CI(\pa_f X, \bundle|_{\pa_f X})\,,
\\
S_{X,\twist}^{\cf}(s) \colon \CI(\pa_f X, \bundle|_{\pa_f X}) \to \C^{n_c^\twist}\,,
\\
S_{X,\twist}^{\fc}(s) \colon \C^{n_c^\twist} \to \CI(\pa_f X, \bundle|_{\pa_f X})\,,
\\
S_{X,\twist}^{\cc}(s) \colon \C^{n_c^\twist} \to \C^{n_c^\twist}\,.
\end{gather*}
For $\Re s < 1/2$, we have that
\begin{align}\label{eq:smatrix-from-resolvent}
S_{X,\twist}(s) = (2s-1) \left. (\rho_f \rho_f')^{-s} (\rho_c \rho_c')^{1-s} \ResTwist(s; z, z') \right|_{\pa_\infty X \times \pa_\infty X}\,.
\end{align}
For $j = 1, \dotsc, n_f$ let $S_{X_{f,j},\twist}(s)$ be the scattering matrix for the funnel end $X_{f,j}$ as described in Section~\ref{sec:model-funnel}.
The scattering matrix for funnel ends $S_{X_f,\twist}(s)$ is diagonal with 
respect to the decomposition of the boundary $\pa_\infty X$ and given by
\begin{align*}
S_{X_f,\twist}(s) &:  \CI(\pa_f X, \bundle|_{\pa_f X}) \to \CI(\pa_f X, \bundle|_{\pa_f X})\,,
\\
S_{X_f,\twist}(s) &\coloneqq S_{X_{f,1},\twist}(s) \oplus \dotsc \oplus S_{X_{f,n_f},\twist}(s)\,.
\end{align*}
As it was already in the case for the resolvent, the scattering matrix $S_{X,\twist}(s)$ is closely related to scattering matrix for the funnel ends, $S_{X_f,\twist}(s)$.
\begin{lemma}\label{lem:Qhash}
    Let $Q^\#(s) \colon \CI(\pa_\infty X, \bundle|_{\pa_\infty X}) \to \CI(\pa_\infty X, \bundle|_{\pa_\infty X})$ be given by the matrix representation
    \begin{align}\label{eq:Qhash}
        Q^\#(s) = \begin{pmatrix}
            Q^\#(s)^{\ff} & Q^\#(s)^{\fc}(s) \\
            Q^\#(s)^{\cf} & Q^\#(s)^{\cc}(s)
        \end{pmatrix}\,,
    \end{align}
    where
    \begin{align*}
        Q^\#(s)^{\ff} &= E_{X_f,\twist}^T(s) ( \eta_3 - \eta_{f,1}) (\id - L(s)\eta_3)^{-1} [\Delta_{X,\twist}, \eta_{f,0}] E_{X_f,\twist}(s)\,,\\
        Q^\#(s)^{\fc} &= E_{X_f,\twist}^T(s) ( \eta_3 - \eta_{f,1}) (\id - L(s)\eta_3)^{-1} [\Delta_{X,\twist}, \eta_{c,0}] E_{X_c,\twist}(s)\,,\\
        Q^\#(s)^{\cf} &= E_{X_c,\twist}^T(s) ( \eta_3 - \eta_{c,1}) (\id - L(s)\eta_3)^{-1} [\Delta_{X,\twist}, \eta_{f,0}] E_{X_f,\twist}(s)\,,\\
        Q^\#(s)^{\cc} &= E_{X_c,\twist}^T(s) ( \eta_3 - \eta_{c,1}) (\id - L(s)\eta_3)^{-1} [\Delta_{X,\twist}, \eta_{c,0}] E_{X_c,\twist}(s)\,.
    \end{align*}
     Then the integral kernel of $Q^\#(s)$ is given by
    \begin{align*}
        Q^\#(s;\omega,\omega') = \lim_{\rho\to 0,\rho' \to 0} (\rho_f \rho_f')^{-s} (\rho_c \rho_c')^{1-s} Q(s;\rho,\omega,\rho',\omega')
    \end{align*}
    for $\Re s < 1/2$.
\end{lemma}
\begin{proof}
    By \eqref{eq:Qs_product} we have that
    \begin{align*}
        Q(s) \varphi = M(s) \eta_3 (\id - L(s)\eta_3)^{-1} L(s) \varphi
    \end{align*}
    for $\varphi \in \CI(\overline{X},\bundle)$ with $\eta_3 \varphi = 0$.
    If $\psi \in \CI(\overline{X}, \bundle)$ with $\eta_1 \psi = 0$, then we can write
    \begin{align*}
        (Q(s)\varphi, \psi)_{L^2} = \left(
        \begin{pmatrix}
            Q^{\ff}(s) & Q^{\fc}(s) \\
            Q^{\cf}(s) & Q^{\cc}(s)
        \end{pmatrix}
        \begin{pmatrix}
            \varphi|_{X_f} \\ \varphi|_{X_c}
        \end{pmatrix}\,,
        \begin{pmatrix}
            \psi|_{X_f} \\ \psi|_{X_c}
        \end{pmatrix}\right)_{L^2(X,\bundle)}\,.
    \end{align*}
    From the definition of $M(s)$ and $L(s)$, we see that for instance
    \begin{align*}
        Q^{\ff}(s) = R_{X_f,\twist}(s) ( \eta_3 - \eta_{f,1}) (\id - 
L(s)\eta_3)^{-1} [\Delta_{X,\twist}, \eta_{f,0}] R_{X_f,\twist}(s)\,.
    \end{align*}
    Using that the integral kernel of $E_{X_f,\twist}(s)^T$ is given by 
\[E_{X_f,\twist}(s)^{T}(\phi,r',\phi') = \lim_{r \to \infty} \rho_f(r)^{-s} 
R_{X_f,\twist}(s;r,\phi,r',\phi')\]
    and the integral kernel of $E_{X_f,\twist}(s)$ is given by \eqref{eq:poisson-model-funnel}, we obtain that
    \begin{align*}
        Q^\#(s; \omega, \omega')^{\ff} = \lim_{\rho\to 0,\rho' \to 0} (\rho_f \rho_f')^{-s} Q^{\ff}(s;\rho,\omega,\rho',\omega')\,.
    \end{align*}
\end{proof}

\begin{prop}\label{prop:smatrix-structure}
    The two scattering matrices, $ S_{X,\twist}(s)$ and $S_{X_f,\twist}(s)$, are related by
    \begin{align}\label{eq:smatrix-structure}
        S_{X,\twist}(s) = S_{X_f,\twist}(s) \oplus 0 + (2s-1) Q^\#(s)\,,
    \end{align}
    where $0\colon \C^{n_c^\twist} \to \C^{n_c^\twist}$ is the zero-map and 
$Q^\#(s)$ is given by Lemma~\ref{lem:Qhash}.
    In particular,
    \begin{align*}
    S_{X,\twist}^{\ff}(s) \in \Psi^{2 \Re s - 1}(\pa_f X, \bundle|_{\pa_f X}), \quad s \not \in \ResSet_{X,\twist} \cup (\N_0 + 1/2)\,.
    \end{align*}
\end{prop}

\begin{proof}
    For $\Re s < 1/2$, this follows directly from the characterization of the 
scattering matrix as a limit of the resolvent, 
\eqref{eq:smatrix-from-resolvent}, Theorem~\ref{thm:resolvent}.
    For $\Re s \geq 1/2$ we use meromorphic continuation. 
    Note that 
$Q^\#(s)^{\ff}$ is smoothing and hence a pseudodifferential operator of order 
$-\infty$.
    The second part then follows from $S_{X,\twist}^{\ff}(s) = S_{X_f,\twist}(s) + Q^\#(s)^{\ff}$ and  \eqref{eq:smatrix-funnel-psido}. 
\end{proof}

\begin{remark}\label{rem:nocuspcontrib}
    The appearance of the map $0 \colon \C^{n_c^\twist} \to \C^{n_c^\twist}$ in \eqref{eq:smatrix-structure} is due to the fact that for $\Re s > 1/2$, we have that
    \[
    \lim_{y \to \infty} \rho_c(y)^{1-s} E_{C_\infty,\twist}(s;y) = 0\,.
    \]
\end{remark}

As in the case of the resolvent, we want to investigate the structure of the 
scattering matrix near a resonance. For this we consider 
\[
\phi_\ell^\# \in \CI(\pa_\infty X, \bundle|_{\pa_\infty X})
\]
defined by
\begin{align*}
\phi_\ell^\#(\omega) \coloneqq \lim_{\rho\to 0} \rho_f^{-s_0} \rho_c^{1-s_0} \phi_\ell(\rho,\omega)\,,
\end{align*}
where $\phi_\ell$ is as in \eqref{eq:resolvent-A_1}. Let
\begin{align*}
\Phi^\#(v,w) \coloneqq \left( \ang{\phi_\ell^\#,v}\right)_{\ell = 1,\dotsc,m_{X,\twist}(s_0)}\,,
\end{align*}
where $\ang{\cdot,\cdot}$ is the bilinear product on $L^2(\pa_\infty X, \bundle|_{\pa_\infty X})$ defined by
\begin{align*}
    \ang{u,v} = \int_{\pa_f X} \ang{u,v}_{\bundle} \, d\sigma_{\pa_f X} + \sum_{j=1}^{n_c^\twist} u_j v_j\,.
\end{align*}
\begin{lemma}\label{lem:smatrix-resonance-decomp}
Let $s_0 \in \C$ with $\Re s_0 < 1$ and $s_0 \not = 1/2$. The scattering matrix has a pole
at $s_0$ if and only if $\ResTwist(s)$ has a pole at $s_0$. In this case we have that
\begin{align*}
S_{X,\twist}(s) &= (\Phi^\#)^T E(s,s_0) \left( \sum_{j=1}^n \left( s(1-s) - s_0(1-s_0) \right)^{-k_j} P_j \right) F(s,s_0) \Phi^\#
\\
&\phantom{=} + H^\#(s,s_0)\,,
\end{align*}
where for some $n, k_j > 0$ with 
\[
\sum_{j=1}^n k_j = m_{X,\twist}(s_0)\,,
\]
for each $j \in \{1,\dotsc, n\}$ the matrices $P_j$ are rank-$1$-projections 
from $\C^{m_{X,\twist}(s_0)}$ to mutually orthogonal subspaces,
$E(\cdot,s_0)$ and $F(\cdot,s_0)$ are holomorphically invertible matrices of dimension $m_{X,\twist}(s_0)$, and
\[
H^\#(\cdot,s_0)\colon L^2(\pa_\infty X, \bundle|_{\pa_\infty X}) \to L^2(\pa_\infty X, \bundle|_{\pa_\infty X})
\]
is holomorphic near $s = s_0$.
\end{lemma}

\begin{proof}
Using \eqref{eq:smatrix-from-resolvent} and \eqref{eq:resolvent-at-resonance} we have that
\begin{align*}
S_{X,\twist}(s) = \sum_{k=1}^{p} \frac{A_k^\#(s_0)}{\left( s(1-s) - s_0(1-s_0)
\right)^k} + H^\#(s,s_0)
\end{align*}
for some (unique) $p\in\N_0$ such that $H^\#(\cdot,s_0)$ is holomorphic. For 
each $k\in\{1,\ldots,p\}$,  the operator~$A_k^\#(s_0)$ is determined by the 
integral kernel
\begin{align*}
A_k^\#(s_0, \omega,\omega') \coloneqq (2s_0-1) \lim_{\rho \to 0}\lim_{\rho' \to 0} (\rho_f \rho_f')^{-s_0} (\rho_c \rho_c')^{1-s_0} A_k(s_0,\rho,\omega,\rho',\omega')\,.
\end{align*}

Recall from \eqref{eq:resolvent-A_k} that
\begin{align*}
A_k(s_0) = \sum_{\ell,m=1}^{m_{X,\twist}(s_0)} a_k^{\ell,m}(s_0) \phi_\ell \otimes \phi_m\,.
\end{align*}
This implies
\begin{align*}
A_k^\#(s_0) &= \sum_{\ell,m=1}^{m_{X,\twist}(s_0)} a_k^{\ell,m}(s_0) \phi^\#_\ell \otimes \phi^\#_m
\\
&= (\Phi^\#)^T a_k(s_0) \Phi^\#\,.
\end{align*}
Above, $a_k(s_0)$ is as in \eqref{eq:resolvent-A_k}. Recall that $a_k(s_0) = a_1(s_0) d(s_0)^{k-1}$, where $d(s_0)$ is nilpotent. Hence, $S_{X,\twist}(s)$ can be written as
\begin{align*}
S_{X,\twist}(s) &= (\Phi^\#)^T a_1(s_0) \left( \sum_{k=0}^{p-1} (s(1-s) - s_0(1-s_0))^{-(k+1)} d(s_0)^k \right) \Phi^\#
\\
&\hphantom{=} + H^\#(s,s_0)
\end{align*}
in a sufficiently small neighborhood of $s_0$.
Denote by $N_k$ a Jordan block of dimension $k$ with eigenvalue $0$. The Jordan normal form of $d(s_0)$ is given by
\begin{align*}
J d(s_0) J^* = 
\begin{pmatrix}
N_{k_1} & 0 & \hdots & 0 
\\
0 & N_{k_2} & & \vdots 
\\ 
\vdots & & \ddots & 0 
\\ 
0 & \hdots & 0 & N_{k_n}
\end{pmatrix}\,,
\end{align*}
where $\sum_{j = 1}^n k_j = m_{X,\twist}(s_0)$ and $J$ is unitary. Using linear 
algebra, we immediately obtain that for each $j\in\{1,\ldots,n\}$, 
\begin{align*}
\sum_{m=0}^{p-1} x^{-(m+1)} N_{k_j}^m = E_{k_j}(x) (x^{-k_j} P_j + \tilde{P}) 
F_{k_j}(x)\,,
\end{align*}
where $E_{k_j}$ and $F_{k_j}$ are polynomials in $x$, and $P_j$, $\tilde{P}$ 
are diagonal matrices, and each $P_j$ has rank one.
Putting $x = s(1-s) - s_0(1-s_0)$ and applying the argumentation above to every 
Jordan block, we obtain matrices $E(s,s_0), F(s,s_0)$ depending polynomially on 
$s$ and mutually orthogonal projections $P_j$ of rank~$1$ such that
\begin{gather*}
\sum_{k=0}^{p-1} (s(1-s) - s_0(1-s_0))^{-(k+1)} d(s_0)^k =
\\
E(s,s_0) \left( \sum_{j=1}^{n} (s(1-s) - s_0(1-s_0))^{-k_j} P_j \right) F(s,s_0) + \tilde{H}(s,s_0)\,,
\end{gather*}
where $\tilde{H}(\cdot,s_0)$ is holomorphic. This proves the claim.
\end{proof}
\subsection{The Scattering Matrix at $s = 1/2$}\label{sec:scat_matrix_12}

\begin{lemma}\label{lem:resonance-onehalf}
The resolvent satisfies
\begin{align}\label{eq:resolvent-onehalf}
\ResTwist(s) = \frac{1}{2s-1} \sum_{k=1}^{m_{X,\twist}(1/2)} \phi_k(s) \ang{\phi_k(s), \cdot} + H(s)\,,
\end{align}
where $H$ is holomorphic near $1/2$,
and, for each $k \in \{1, \ldots, m_{X,\twist}(1/2)\}$, the function
\begin{align}\label{eq:phik_asymptotics}
\phi_k \in \rho_f^s\rho_c^{s-1}\CI(X,\bundle)
\end{align}
satisfies 
\begin{align*}
(\LapTwist - \tfrac14)\phi_k(\tfrac12) = 0\,.
\end{align*}
\end{lemma}

\begin{proof}
We note that $\Im( s^2 - s) = \Im\left( (s - 1/2)^2 \right)$. Let $\psi \in \CcI(X, \bundle)$. Using the self-adjointness of $\LapTwist$, we obtain the estimate
\begin{align*}
\abs*{ \left( (\LapTwist - s(1-s)) u, u \right)_{L^2} } &\geq \abs*{\Im  \left( (\LapTwist - s(1-s)) u, u \right)_{L^2} } 
\\
&= \abs*{\Im(s^2 - s)} \norm{u}^2_{L^2} 
\\
&= \abs*{\Im\left( \left(s - \tfrac12\right)^2\right)} \norm{u}^2_{L^2} \,.
\end{align*}
Therefore, we have
\begin{align}\label{eq:resolvent-bound}
\norm{\ResTwist(s)} \leq \abs*{\Im\left( \left(s - \tfrac12\right)^2\right)}^{-1}\,.
\end{align}
Hence, the order of the resonance at $s = 1/2$ is at most $2$. This implies  
that
\begin{align*}
\ResTwist(s) = \frac{A}{(2s-1)^2} + \frac{B}{2s-1} + h(s)\,,
\end{align*}
where $h$ is holomorphic near $1/2$, and $A$ and $B$ are suitable operators, 
independent of~$s$.

Using the resolvent equation, we see that every element $u$ in the range of $A$ and $B$ satisfies $(\LapTwist - 1/4)u = 0$. We note that  \eqref{eq:resolvent-bound} implies that  $A\colon L_\cpt^2(X,\bundle) \to L^2(X,\bundle)$.
Hence, the range of $A$ consists of eigenfunctions of $\LapTwist$ with eigenvalue $1/4$.
By Proposition~\ref{prop:no-eigenvalues} there are no eigenfunctions if $X$ has infinite volume, hence $A = 0$.

By the definition of a multiplicity, we have that $\rank B = m_{X,\twist}(1/2)$.
Using the decomposition of the resolvent from \eqref{eq:resolvent-A_1}, we can write $\ResTwist(s)$ as $(2s-1)^{-1} B(s) + H(s)$, where
\begin{align*}
    B(s) = \sum_{\ell,m=1}^{m_{X,\twist}(1/2)} a_1^{\ell, m} (s) \tilde\phi_\ell(s) \ang{\tilde\phi_m(s), \cdot}
\end{align*}
for some symmetric invertible matrix $a_1(1/2) = (a_1^{\ell,m}(s_0))_{\ell,m=1}^{m_{X,\twist}(1/2)}$, $\tilde\phi_k \in \rho_f^s \rho_c^{s-1} \CI(X,\bundle)$ and $H(s)$ is holomorphic near $s = 1/2$.
Since the resolvent at $1/2$ is self-adjoint and non-negative, $a_1(1/2)$ is a 
positive matrix. Therefore we can find a matrix
$(d_{k,\ell})_{k,\ell=1}^{m_{X,\twist}(1/2)}$ such that
\begin{align}\label{eq:B(1/2)}
    B(1/2) = \sum_{k=1}^{m_{X,\twist}(1/2)} \phi_k \ang{\phi_k, \cdot}\,,
\end{align}
where $\phi_k = \sum_{\ell = 1}^{m_{X,\twist}(1/2)} d_{k,\ell} \tilde\phi_\ell(1/2)$ for $k = 1, \dotsc, m_{X,\twist}(1/2)$.
We have that $B(1/2) = B$ and hence
\begin{align*}
(\LapTwist - \tfrac14) \phi_k(\tfrac12) = 0\,.
\end{align*}
\end{proof}

\begin{proof}[Proof of Theorem~\ref{thm:smatrix-onehalf}]
From Proposition~\ref{prop:poisson-asymptotics} and 
Definition~\ref{def:smatrix}, we obtain that
\begin{align*}
(2s-1) E_{X,\twist}(s) u \sim \rho_f^{1-s} \rho_c^{-s} u + \rho_f^s \rho_c^{s-1} S_{X,\twist}(s) u\,.
\end{align*}
At first glance, this does not make any sense for $s = 1/2$, but we will see that $E_{X,\twist}(s)$ has a simple pole at $s = 1/2$ and hence $(2s-1) E_{X,\twist}(s) \not = 0$ for $s = 1/2$.

By Theorem~\ref{thm:resolvent}, we have the decomposition
\begin{align*}
    \ResTwist(s) = \tilde{M}_i(s) + M_f(s) + M_c(s) + Q(s)
\end{align*}
and we recall that $\tilde{M}_i(s)$ and $M_f(s)$ are holomorphic near $s = 1/2$. We write the remainder term $Q(s)$ as
\begin{align*}
Q(s) = (2s-1)^{-1} \tilde{Q} + Q_\hol(s)\,,
\end{align*}
where $Q_\hol$ is holomorphic near $s = 1/2$.
By \eqref{eq:poisson-model-cusp} and \eqref{eq:resolvent-mc}, the term $M_c(s)$ is given by
\begin{align*}
(2s-1) M_c(s) = (1 - \eta_{c,0}) \rho_c^{s-1} (\rho_c')^{-s} \id_{\pa_c \bundle} (1 - \eta_{c,1}) + (2s-1) M_{\hol}^c(s)\,,
\end{align*}
with $M_{\hol}^c(s)$ being holomorphic near $s = 1/2$. If we set
\begin{align*}
    \tilde{M}_c = (1 - \eta_{c,0}) (\rho_c \rho_c')^{-1/2} \id_{\pa_c \bundle}(1 
- \eta_{c,1}) \,,
\end{align*}
then we have that
\begin{align*}
\ResTwist(s) = (2s-1)^{-1} \left( \tilde{Q} + \tilde{M}_c \right) + H(s)\,
\end{align*}
and $H(s)$ is holomorphic near $s = 1/2$.
Recall also that 
\[
S_{X,\twist}(s) = (S_{X_f,\twist}(s) \oplus 0) + 
(2s-1)Q^\#(s)\,,
\]
where
\begin{align*}
Q^\#(s;\omega,\omega') = \left. (\rho_f \rho_f')^{-s} (\rho_c \rho_c')^{1-s} Q(s;\rho,\omega,\rho',\omega')\right|_{\pa_\infty X \times \pa_\infty X}\,.
\end{align*}
Pick $\tilde{Q}^\#$ such that
\begin{align*}
    (2s - 1) Q^\#(s) = \tilde{Q}^\# + Q^\#_{\hol}(s)\,,
\end{align*}
where $Q^\#_{\hol}(s)$ is holomorphic near $s = 1/2$.
This implies that
\begin{align*}
    \tilde{Q}^\# = \left. (\rho_f \rho_f')^{-1/2} (\rho_c \rho_c')^{1/2} \tilde{Q}(\rho,\omega,\rho',\omega')\right|_{\pa_\infty X \times \pa_\infty X}\,.
\end{align*}
From the Fourier decomposition of $S_{X_f,\twist}(s)$, we see that $S_{X_f,\twist}(1/2) = - \id$. This implies that
\begin{align*}
P &\coloneqq \frac{1}{2} \left( S_{X,\twist}(\tfrac12) + \id \right)
\\
&= \frac{1}{2} \left( (0 \oplus \id_{\pa_c \bundle}) + \tilde{Q}^\#\right)
\end{align*}
is a compact operator. Using 
 \eqref{property_of_scattering_under_adjoint_and_transpose} and Proposition \ref{prop:intertwine}, 
 we calculate $P^2 = P$ and $P^* = P$.

The residue of the resolvent at $s = 1/2$ is given by
\begin{align*}
B(\tfrac12) = \tilde{Q} + \tilde{M}_c\,.
\end{align*}
This implies that
\begin{align*}
    \left. (\rho_f \rho_f')^{-1/2} (\rho_c \rho_c')^{1/2} B(\tfrac12) \right|_{\pa_\infty X \times \pa_\infty X} &= \tilde{Q}^\# + (0 \oplus \id_{\pa_c \bundle}) 
\\
&= 2P\,.
\end{align*}
With $\phi_k$ given by \eqref{eq:resolvent-onehalf}, we set
\begin{align}\label{def:phi_k}
\phi_k^\#(s) \coloneqq \left. \rho_f^{-s} \rho_c^{1-s} \phi_k(s)\right|_{\pa_\infty X}\,,
\end{align}
which defines a function $\phi_k^\#(s) \in \CI(\pa_\infty X, \bundle)$ by \eqref{eq:phik_asymptotics}.
We note that $\phi_k^\#(s)$ is holomorphic in $s$ for $s$ close to $1/2$. 
Further, the functions $\phi_k^\#(1/2)$ are linearly independent since, 
otherwise, a non-trivial linear combination would lead to an $L^2$-integrable 
solution of the eigenvalue equation in contradiction to $\LapTwist$ having no 
eigenvalues at $\lambda = 1/4$.

From \eqref{eq:B(1/2)} and \eqref{def:phi_k} we obtain that
\begin{align*}
\left. (\rho_f \rho_f')^{-1/2} (\rho_c \rho_c')^{1/2} B(\tfrac12) \right|_{\pa_\infty X \times \pa_\infty X} = \sum_{k=1}^{m_{X,\twist}(1/2)} \phi_k^\#(\tfrac12) \ang{\phi_k^\#(\tfrac12), \cdot}\,.
\end{align*}
Thus the restriction of $B(1/2)$ to the boundary at infinity still has rank 
$m_{X,\twist}(1/2)$. This finishes the proof.
\end{proof}

\begin{remark}
The proof also shows that
\begin{align*}
S_{X,\twist}(\tfrac12) = -\id + \sum_{k=1}^{m_{X,\twist}(1/2)} \phi_k^\#(\tfrac12) \ang{\phi_k^\#(\tfrac12), \cdot}\,.
\end{align*}
\end{remark}

\subsection{Scattering Poles}\label{sec:scattering_poles}
Let $s_0 \in \C$ be a resonance, let $\varepsilon > 0$ and let $\gamma_{s_0,\eps}$ be the path 
\begin{equation}\label{def:small_path}
[0,1] \ni t \mapsto s_0 + \eps e^{2\pi i t}.\end{equation}
 We suppose that $\eps$ is small enough such that there is no other resonance 
inside $\gamma_{s_0,\eps}$ rather than $s_0$. Recall that the 
resonance multiplicity of $s_0$ is given as
\begin{align*}
m_{X,\twist}(s_0) \coloneqq \rank \int_{\gamma_{s_0,\eps}} \ResTwist(t)\,dt\,, \quad s_0 \not = \frac12\,.
\end{align*}
The analogues of resonances for a scattering matrix are \emph{scattering poles}. The definition of the multiplicity of a scattering pole is more involved.

\label{GS-theory}
We start with briefly recalling some definitions from the Gohberg--Sigal theory 
\cite{GohbergSigal}: let $\banach$ be a Banach space and let $\lambda_0 \in \C$. 
We further denote by $\mathfrak{A}$ the algebra of all linear bounded operators 
from  $\banach$ to $\banach$. We denote by $\mathcal{M}(\lambda_0)$ 
the germ 
of $\mathfrak{A}$-valued functions that are holomorphic in some punctured 
neighborhood of~$\lambda_0$ and have either a pole or a removable singularity at 
$\lambda_0$. In any concrete situation we will pick a suitable neighborhood.

Let $B \in \mathcal{M}(\lambda_0)$ be holomorphic at least in $\Omega_B 
\setminus \{\lambda_0\}$, 
where $\Omega_B$ is some open neighborhood of $\lambda_0$, and suppose that there exists a function $\psi\colon \Omega_B \to \banach$ such that $\psi(\lambda_0) \neq 0$, the functions $\psi$ and $B \psi$ 
are holomorphic at~$\lambda_0$; moreover, we suppose that $B \psi(\lambda_0) = 
0$. We refer to $\psi(\lambda_0)$ as a \emph{root vector} and to $\psi$ as a 
\emph{root function} of $B$ at~$\lambda_0$. The \emph{rank} of a root vector 
$\psi(\lambda_0)$, further denoted as $\rank(\psi(\lambda_0))$, is the maximal 
order of vanishing of $B(\lambda) \phi(\lambda)$  at $\lambda = \lambda_0$
among all root functions~$\phi$ with $\phi(\lambda_0) = \psi(\lambda_0)$. If 
these orders of vanishing are unbounded, we define $\rank(\psi(\lambda_0)) 
\coloneqq \infty$. The set of all root vectors of~$B$ at~$\lambda_0$ is a vector 
space. 
We refer to its closure in~$\banach$ as the \textit{kernel} of $B(\lambda_0)$ and denote it by $\ker B(\lambda_0)$. 
In what follows, we suppose that $m \coloneqq \dim \ker B(\lambda_0) < \infty$ 
and $\rank(v) < \infty$ for all $v \in \ker B(\lambda_0)$. We define a basis, 
$\{v^{(1)}, \dots, v^{(m)}\}$, of $\ker B(\lambda_0)$ as follows: the rank of 
$v^{(1)}$ equals the maximal rank of all root vectors corresponding to 
$\lambda_0$ and the rank of~$v^{(j)}$ for $j = 2, \ldots, m$ is the maximal rank 
of root vectors in some direct complement of the span $\{v^{(1)}, \dotsc, 
v^{(j-1)}\}$. Let $r_j \coloneqq \rank v^{(j)}$. 
We set
\begin{align*}
N_{\lambda_0}(B) \coloneqq \sum_{j=1}^m r_j\,.
\end{align*}
We also recall (see, e.g., \cite[Definition 6.6]{Borthwick_book}) that a set 
of bounded operators $A(\lambda)$ from $\banach$ to $\banach$, parametrized by 
$\lambda \in U \subset \C$, is a \textit{finitely meromorphic family} if at each 
point $\lambda' \in U$, we have a Laurent series representation,
\[
A(\lambda) = \sum_{k=-m}^{\infty} (\lambda - \lambda')^k A_k,
\]
converging (in the operator topology) in some neighborhood of $\lambda'$, where for $k<0$, the coefficients $A_k$ are finite rank operators.

The main result of Gohberg--Sigal~\cite[Theorem~2.1]{GohbergSigal} is the 
following argument principle:  Let $B \in \mathcal{M}(\lambda_0)$ be such that 
$B$ is invertible in some neighborhood of $\lambda_0$. Suppose that 
$B$ and $B^{-1}$ are finitely meromorphic families of operators in this 
neighborhood of $\lambda_0$. Suppose that all points inside a sufficiently 
small 
contour,~$\gamma$, around $\lambda_0$ (except for, maybe, $\lambda_0$ itself) 
are regular for both $B$ and $B^{-1}$. Additionally, suppose that the 
non-singular part of $B$ at $\lambda_0$ has index zero. Then 
\begin{align}\label{eq:gohberg-sigal}
N_{\lambda_0}(B) - N_{\lambda_0}(B^{-1}) = \frac{1}{2\pi i} \Tr \int_{\gamma} B(\lambda)^{-1} B'(\lambda) \, d\lambda\, .
\end{align}
If for such $B \in \mathcal{M}(\lambda_0)$ we define
\begin{align}\label{eq:definition_of_m}
 M_{\lambda_0}(B) \coloneqq \frac{1}{2\pi i} \Tr \int_{\gamma} B(\lambda)^{-1} B'(\lambda) \, d\lambda\,,
\end{align}
then for all $B_1, B_2 \in 
\mathcal{}(\lambda_0)$ satisfying the conditions above we have 
\begin{align}\label{eq:M-cyclic}
M_{\lambda_0}(B_1 B_2) = M_{\lambda_0}(B_1) + M_{\lambda_0}(B_2)\,.
\end{align}
See \cite[Theorem~5.2]{GohbergSigal}.

From \eqref{eq:smatrix-funnel} and Proposition~\ref{prop:smatrix-structure}, we 
obtain that $S_{X,\twist}(s)$ has poles of infinite rank at $s = 1/2 + \N_0$. 
Hence, we define the operator
\begin{align*}
G(s) &\colon  \CI(\pa_\infty X, \bundle|_{\pa_\infty X}) \to \CI(\pa_\infty X, \bundle|_{\pa_\infty X})\,,
\\
G(s) & \coloneqq ( \gammafunc(s+\tfrac12) \id_{\CI(\pa_f X,\bundle|_{\pa_f X})} ) \oplus \id_{\CI(\pa_c X,\bundle|_{\pa_c X})}\,.
\end{align*}
We want to normalize the scattering matrix such that it is a bounded operator 
for all $s \not \in \ResSet_{X, \twist} \cup (1/2 + \N_0)$. 
Denote by $\Lambda_{\pa_f X}$ the square-root of the Laplacian with respect to the bundle metric -- or any other  invertible elliptic operator $\Lambda_{\pa_f X} \in \Psi^1(\pa_f X, \bundle|_{\pa_f X})$. Set 
\begin{align*}
\Lambda(s) &\colon  \CI(\pa_\infty X, \bundle|_{\pa_\infty X}) \to \CI(\pa_\infty X, \bundle|_{\pa_\infty X})\,,
\\
\Lambda(s) &= \Lambda_{\pa_f X}^{-s+1/2} \oplus \id_{\CI(\pa_c X,\bundle|_{\pa_c X})}\,.
\end{align*}
Note that $\Lambda(s)$ and $G(s)$ commute, and we have that $\Lambda(1-s)^{-1} = \Lambda(s)$.
It follows from Proposition~\ref{prop:smatrix-structure} that
\begin{align}\label{def:Stilde}
    \tilde{S}_{X,\twist}(s) \coloneqq G(s) \Lambda(s) S_{X,\twist}(s) \Lambda(1-s)^{-1} G(1-s)^{-1}
\end{align}
is a meromorphic family of pseudodifferential operators of order $0$ with poles of finite rank. Note that both $G(s)$ and $G(1-s)^{-1}$ are invertible away from $s \in \frac{1}{2} \pm \N$. Moreover, we have that
\begin{align*}
\tilde{S}_{X,\twist}(1-s) = \tilde{S}_{X,\twist}(s)^{-1}
\end{align*}
and  that $\tilde{S}(s)$ is a Fredholm operator by Proposition~\ref{prop:smatrix-structure} and the invertibility of $S_{X_f,\twist}(s)$. We only have to consider the $S_{X,\twist}^{\ff}$, the other entries are finite rank.
As for $S_{X,\twist}(s)$, we can write $\tilde{S}_{X,\twist}(s)$ as a $2 \times 2$ matrix,
\begin{align}\label{eq:tildeS_matrix}
    \tilde{S}_{X,\twist}(s) =
    \begin{pmatrix} 
        \tilde{S}_{X,\twist}^{\ff}(s) & \tilde{S}_{X,\twist}^{\fc} 
        \\
        \tilde{S}_{X,\twist}^{\cf}(s) & \tilde{S}_{X,\twist}^{\cc}
    \end{pmatrix}\,,
\end{align}
where
\begin{align*}
    \tilde{S}_{X,\twist}^{\ff}(s) &\coloneqq \frac{\gammafunc(s+\tfrac12)}{\gammafunc(\tfrac32-s)} \Lambda_{\pa_f X}^{-s+1/2} S_{X,\twist}^{\ff}(s) \Lambda_{\pa_f X}^{-s+1/2}\,,\\
    \tilde{S}_{X,\twist}^{\cf}(s) &\coloneqq \gammafunc(\tfrac32 - s)^{-1} S_{X,\twist}^{\cf}(s) \Lambda_{\pa_f X}^{-s+1/2}\,,\\
    \tilde{S}_{X,\twist}^{\fc}(s) &\coloneqq \gammafunc(s+\tfrac12) \Lambda_{\pa_f X}^{-s+1/2} S_{X,\twist}^{\fc}(s) \,,\\
    \tilde{S}_{X,\twist}^{\cc}(s) &\coloneqq S_{X,\twist}^{\cc}(s) \,.
\end{align*}
The multiplicity of a scattering pole $s_0 \in \C$ is defined as
\begin{align}\label{eq:def_nu}
    \nu_{X,\twist}(s_0) \coloneqq -M_{s_0}(\tilde{S}_{X,\twist}) =  - \frac{1}{2\pi i} \Tr \int_{\gamma} \tilde{S}_{X,\twist}(s)^{-1} \frac{d}{ds} \tilde{S}_{X,\twist}(s) \, ds\,.
\end{align}
By \eqref{eq:M-cyclic} it follows that $\nu_{X,\twist}(s)$ is independent of the 
specific choice of the operator~$\Lambda_{\pa_f X}$.

\begin{lemma}\label{lem:simplify-N}
For $s_0 \in \ResSet_{X,\twist}$ with $\Re s_0 < 1, s_0 \not = 1/2$, we have that
\begin{align*}
N_{1-s_0}(\tilde{S}_{X,\twist}) = N_{1-s_0}(\Lambda S_{X,\twist} \Lambda)\,.
\end{align*}
Moreover, for a resonance $s_0 \in \ResSet_{X,\twist}$  there exists $n^\# > 0$ and $k_j^\# \in \Z$, such that we have the decomposition near $s_0 \in \ResSet_{X,\twist}$, 
\begin{align*}
\Lambda(s) S_{X,\twist}(s) \Lambda(s) = G_1(s) \left( \tilde{P}_0(s) + \sum_{j=1}^{n^\#} (s-s_0)^{-k^\#_j} P_j \right) G_2(s)\,,
\end{align*}
where
$G_1,G_2$ are holomorphically invertible near $s_0 \in \ResSet_{X,\twist}$ and
\begin{align*}
\tilde{P}_0(s) = \begin{cases}
(s - s_0) P_0, & s_0 \in \ResSet_{X,\twist} \cap (\frac12 - \N)\,,
\\
P_0, & s_0 \in \ResSet_{X,\twist} \setminus (\frac12 - \N)
\end{cases}
\end{align*}
and $P_0$ is a projection. 
\end{lemma}

\begin{proof}
The first part of the statement for $s_0 \not \in \frac{1}{2} - \N$ follows from 
\eqref{def:Stilde} and the remark afterwards. 
Now let us consider $s_0 \in \frac{1}{2}-\N$ for which we follow \cite[Lemma 
8.12]{Borthwick_book}. 
We set
\[
T(s) \coloneqq \tilde{S}^{\cc}(s) - \tilde{S}^{\cf}(s) \tilde{S}^{\ff}(s)^{-1} \tilde{S}^{\fc}(s)
\]
and note that it is well-defined near $1-s_0$.
We can then write
\[
\tilde{S}_{X}(s)=\left(\begin{array}{cc}
\id & 0 \\
\tilde{S}^{\cf}(s) \tilde{S}^{\ff}(s)^{-1} & \id 
\end{array}\right)\left(\begin{array}{cc}
\id  & 0 \\
0 & T(s)
\end{array}\right)\left(\begin{array}{cc}
\tilde{S}^{\ff}(s) & \tilde{S}^{\fc}(s) \\
0 & \id 
\end{array}\right)
\]
The first and last factors on the right hand side of the previous equation 
are both invertible near $1-s_0$. Together with 
\cite[Section~1]{GohbergSigal} this implies that 
\[
N_{1-s_0}\left(\tilde{S}_{X}\right)= N_{1-s_0}\left(  \left(\begin{array}{cc}
\id  & 0 \\
0 & T
\end{array}\right) \right) = N_{1-s_0}(T).
\]
Moreover
\begin{align*}
    \Lambda(s) S_{X,\twist}(s) \Lambda(s) &=
\left(\begin{array}{cc}
\id & 0 \\
\Gamma\left(s+\frac{1}{2}\right) \tilde{S}^{\cf}(s) \tilde{S}^{\ff}(s)^{-1} & \id
\end{array}\right) \\
&\phantom{=} \times \left(\begin{array}{cc}
\frac{\Gamma\left(\frac{3}{2}-s\right)}{\Gamma\left(s+\frac{1}{2}\right)} \id & 0 \\
0 & T(s)
\end{array}\right)
\left(\begin{array}{cc}
\tilde{S}^{\ff}(s) & \frac{1}{\Gamma\left(\frac{3}{2}-s\right)} \tilde{S}^{\fc}(s) \\
0 & \id
\end{array}\right).
\end{align*}
We note that  the first and third factors of the right hand side of the equality above are invertible near $s=1-s_0.$  Hence,
\[
N_{1-s_0}\left(\Lambda S_{X} \Lambda\right) =N_{1-s_0} \left( \left(\begin{array}{cc}
\frac{\Gamma\left(\frac{3}{2}-s\right)}{\Gamma\left(s+\frac{1}{2}\right)} \id & 0 \\
0 & T(s)
\end{array}\right) \right)
\]
Since $1+s_0 \in -\N_0$, the function $\Gamma(\frac{3}{2}-s)$ is 
singular at $s=1-s_0$ and hence 
$\Gamma\left(\frac{3}{2}-s\right)/\Gamma\left(s+\frac{1}{2}\right)$ is singular 
as well and thus has no root vectors. Therefore 
\[
N_{1-s_0} \left( \left(\begin{array}{cc}
\frac{\Gamma\left(\frac{3}{2}-s\right)}{\Gamma\left(s+\frac{1}{2}\right)} \id & 0 \\
0 & T(s)
\end{array}\right) \right) =N_{1-s_0}(T)
\]
 which implies $N_{1-s_0}\left(\Lambda S_{X} \Lambda\right) = N_{1-s_0}(T)$ and proves the result.

 The second part of the statement follows from the application of the Gohberg-Sigal Logarithmic Residue Theorem, \eqref{eq:gohberg-sigal}, to $\Lambda S_{X,\twist} \Lambda$.
 \end{proof}

We have that 
\[
N_{1-s_0}(\Lambda S_{X,\twist} \Lambda) = \sum_{j \setmid k_j^\# > 0} k_j^\#\,.
\]
Lemma~\ref{lem:smatrix-resonance-decomp} implies that
\begin{align*}
\sum_{j\colon k_j^\# > 0} k^\#_j \leq \sum_{j=1}^n k_j\,.
\end{align*}

\begin{prop}[Relation between scattering poles and resonances]\label{prop:sp-res}
For ${s_0 \in \C}$ with $\Re s_0 \leq 1$ we have 
\begin{align*}
\nu_{X,\twist}(s_0) = m_{X,\twist}(s_0) - m_{X,\twist}(1-s_0)\,.
\end{align*}
\end{prop}

\begin{proof}
    First, we note that ${S_{X,\twist}(1/2) = \tilde{S}_{X,\twist}(1/2)}$ is unitary, and therefore $\nu_{X,\twist}(1/2) = 0$ by \eqref{eq:def_nu}.
    Moreover, $m_{X,\twist}(1/2) - m_{X,\twist}(1 - 1/2) = 0$, which implies the claimed equality for $s_0 = 1/2$.
Therefore it suffices to consider a resonance $s_0 \in \C$ with $\Re s_0 < 
1$ and $s_0 \not = 1/2$.
By \eqref{eq:gohberg-sigal}, we have that
\begin{align*}
    \nu_{X,\twist}(s_0) &= -M_{s_0}(\tilde{S}_{X,\twist}) 
   = N_{1-s_0}(\tilde{S}_{X,\twist}) - N_{s_0}(\tilde{S}_{X,\twist})\,.
\end{align*}
It remains to show that $m_{X,\twist}(s_0) = N_{1-s_0}(\tilde{S}_{X,\twist})$.
Note that the inequality $m_{X,\twist}(s_0) \geq N_{1-s_0}(\tilde{S}_{X,\twist})$ follows 
from
\begin{align*}
N_{1-s_0}(\Lambda S_{X,\twist} \Lambda) = \sum_{j\colon k_j^\# > 0} k^\#_j \leq \sum_{j=1}^n k_j = m_{X,\twist}(s_0)\,.
\end{align*}
Since the operator $\Phi^\#$ in Lemma~\ref{lem:smatrix-resonance-decomp} might
not have full rank, we cannot directly deduce equality.
To prove $m_{X,\twist}(s_0) \leq N_{1-s_0}(\tilde{S}_{X,\twist})$, we have to 
use \eqref{eq:resolvent-difference-smatrix}. Assume that $s_0(1-s_0) $ does not belong to the discrete spectrum of $\LapTwist$. Then we have that $\Re s_0 < 1/2$ and $S_{X,\twist}(s)$ is 
holomorphic near $1 - s_0$ by definition. Thus, $\tilde{S}_{X,\twist}(s)$ is 
holomorphic near $1 - s_0$ and hence $N_{s_0}(\tilde{S}_{X,\twist}) = 0$, which 
follows by using 
that $\tilde{S}_{X,\twist}(s) \tilde{S}_{X,\twist}(1-s) = \id$.
By Lemma~\ref{lem:simplify-N} and \eqref{eq:resolvent-difference-smatrix}, we 
have that
\begin{align*}
\ResTwist(s) & = \ResTwist(1-s) + (2s-1) E_{X,\twist}(1-s) \Lambda(s)^{-1} 
G_1(s) \\
& \quad \times \left( \tilde{P}_0(s) + \sum_{j=1}^{n^\#} (s-s_0)^{-k^\#_j} P_j 
\right) G_2(s) \Lambda(s)^{-1} E_{X,\twist}(1-s)^T\,.
\end{align*}
Since all terms except for the factors $(s- s_0)^{-k^\#_j}$ are holomorphic and 
the $P_j$ have rank $1$, we have an upper bound for the rank of the residue 
$A_1$ of $\ResTwist(s)$ in \eqref{eq:resolvent-at-resonance},
\begin{align*}
m_{X,\twist}(s_0) = \rank A_1 \leq \sum_{j\colon k_j^\# > 0} k^\#_j = N_{1-s_0}(\Lambda S_{X,\twist} \Lambda)\,.
\end{align*}
If $s_0(1-s_0)$ belongs to the discrete spectrum of $\LapTwist$, we consider separately 
the two cases $\Re s_0 > 1/2$ and $\Re s_0 < 1/2$. 

Let $\Re s_0 > 1/2$. The resolvent estimate implies that the order of the 
resonance at $s_0$ is $1$. Straightforward argumentation shows that $A_1$ 
in \eqref{eq:resolvent-at-resonance} is the projection onto the eigenspace. Let 
$(\phi_i)_{i=1}^{m_{X,\twist}(s_0)}$ be an orthonormal basis of the eigenspace 
and set
\begin{align*}
\phi_i^\# \coloneqq \lim_{\rho \to 0} \rho_f^{-s} \rho_c^{1-s} \phi_i \in \CI(\pa_\infty X, \bundle|_{\pa_\infty X})\,.
\end{align*}
The functions $\phi_i^\#$, $i\in\{1,\ldots,m\}$, are linearly independent, by a 
straightforward contradiction argument using Proposition~\ref{prop:unique-cont}. 
The Laurent expansion of $S_{X,\twist}(s)$ takes the form
\begin{align*}
S_{X,\twist}(s) = - (s- s_0)^{-1} \sum_{i = 1}^{m_{X,\twist}(s_0)} \phi_i^\# \ang{\phi_i^\#, \cdot} + H_1(s)\,,
\end{align*}
where $H_1$ is holomorphic near $s=s_0$. Hence, $\tilde{S}_{X,\twist}^{-1}$ has $m_{X,\twist}(s_0)$ independent root vectors of rank $1$ at $s = s_0$.

Let $s_0(1-s_0) \in \sigma_d(\LapTwist)$ with $\Re s_0 < 1/2$. For 
$i\in\{1,\ldots,m\}$, let $\phi_i$ and $\phi_i^\#$ be as above. Denote the span 
of $\{\phi_i\}_{i=1}^m$ by $W$. Since $\Re s_0 < 1/2$, we have that $W \subset 
\rho_f^{1-s_0} \rho_c^{-s_0} \CI(\overline{X}, \bundle)$. Using Taylor expansion 
of $\rho_f^s \rho_c^{1-s}$ as a function of $s(1-s)$ near $s_0(1-s_0)$, we have 
that
\begin{align*}
\ran A_1(s_0) \subset \sum_{k=0}^{p-1} \rho_f^{s_0} \rho_c^{s_0-1} \CI(\overline{X}, \bundle)\,.
\end{align*}
Using the unique continuation again, it follows that $\ran A_1(s_0)$ and $W$ 
are disjoint. Therefore there exists a decomposition $\rho^{-1} L^2 = W 
\oplus W'$ with
$\ran A_1(s_0) \subset W'$. Denote by $\Pi$ the projection onto $W'$ with $\ker \Pi = W$. We have that $\phi_i \Pi = 0$ and $\Pi A = A$. The Laurent expansion of $\ResTwist(1-s)$ near $s_0$ is given by
\[
\ResTwist(s) = (s-s_0)^{-1} R_{-1} + R_{\hol}(s)\,,
\]
where $R_{\hol}$ is holomorphic. To calculate the residue, we note that
\[
s(1-s) - s_0(1-s_0) = - (s-s_0) (2s_0 - 1 + (s - s_0))
\]
and hence
\begin{align*}
R_{-1} = -\res_{s = s_0} \ResTwist(s) = (2s_0 - 1)^{-1} \sum_i \phi_k \ang{\phi_k, \cdot}\,.
\end{align*}
We define the Laurent expansions
\begin{align*}
(2s-1) E_{X,\twist}(1-s) \Lambda(s)^{-1} G_1(s) \eqqcolon \sum_{l = -1}^\infty 
(s-s_0)^l E_l\,,
\\
G_2(s) \Lambda(s)^{-1} E_{X,\twist}(1-s)^T \eqqcolon \sum_{m=-1}^\infty 
(s-s_0)^m F_m\,.
\end{align*}
The principal parts of these Laurent expansions are given by
\begin{align*}
E_{-1} = \sum_i \phi_i \ang{e_i, \cdot}\,,
\\
F_{-1} = \sum_i f_i \ang{\phi_i, \cdot}\,,
\end{align*}
for some $e_i, f_i \in \CI(\overline{X}, \bundle)$. Consequently,
\begin{align*}
\Pi R_{-1} = 0\,,\quad \Pi E_{-1} = 0\,,\quad\text{and}\quad 
F_{-1} \Pi^T = 0\,.
\end{align*}
The residue at $s_0$ can be calculated as
\begin{align*}
A_1(s_0) = \res_{s_0} \ResTwist = R_{-1} + \sum_{k_j^\# + l + m = -1} E_l P_j F_m\,.
\end{align*}
Conjugating by $\Pi$ yields
\begin{align*}
A_1(s_0) = \Pi A_1(s_0) \Pi^T = \sum_{j\colon k_j > 0} \sum_{l = 0}^{k_j^\# -1} \Pi E_l P_j F_{k_j-1-l} \Pi^T\,.
\end{align*}
Hence,
\begin{align*}
m_{X,\twist}(s_0) = \rank A_1(s_0) \leq \sum_{j\colon k_j^\# > 0} k_j^\# = N_{1-s_0}(\tilde{S}_{X,\twist})\,.
\end{align*}
\end{proof}

\subsection{Relative Scattering Matrix}\label{sec:relative_scattering_matrix}
The relative scattering matrix, defined by
\begin{align}\label{def:relative_scat_matrix}
S_{X,\twist}^\rel(s) \coloneqq \left(S_{X_f,\twist}(s)^{-1} \oplus 
(-\id)\right) S_{X,\twist}(s)\,,
\end{align}
is a smoothing operator on $\CI(\pa_\infty X,\bundle|_{\pa_\infty X})$. 
Therefore it makes sense to define the \emph{relative scattering 
determinant}  
\begin{align}\label{def:scattering_determinant}
\tau_{X,\twist}(s) \coloneqq \det S_{X,\twist}^\rel(s)\,.
\end{align}
The relation \eqref{eq:scattering-matrix-intertwine} implies that
\begin{align}\label{eq:sd-functional}
\tau_{X,\twist}(s) \tau_{X,\twist}(1-s) = 1
\end{align}
and thus
\begin{align}\label{eq:sd-unitary}
\abs{\tau_{X,\twist}(s)} = 1 \qquad\text{for $\Re s = \frac12$}\,.
\end{align}
Let
\[
E_2(s) \coloneqq (1-s) \exp\left(s+ \frac{s^2}{2}\right)\,.
\]
By \cite[Theorem~B]{DFP} and \cite[Theorem~2.6.5]{Boas_entire}, the Weierstrass product
\begin{align}\label{eq:weierstrass-product}
    \ProdTwist(s) \coloneqq s^{m_{X,\twist}(0)} \prod_{\mu \in \ResSet_{X,\twist}\setminus \{0\}} E_2\left(\frac{s}{\mu}\right)
\end{align}
is well-defined and holomorphic of order $2$.

The Weierstrass product~$\ProdModel(s)$ for~$X_f$ is defined analogously, only exchanging~$X$ for~$X_f$ in~\eqref{eq:weierstrass-product}, i.e.,
\begin{align}\label{eq:weierstrass-product_funnel}
    \ProdModel(s) \coloneqq s^{m_{X_f,\twist}(0)} \prod_{\mu \in \ResSet_{X_f,\twist}\setminus \{0\}} E_2\left(\frac{s}{\mu}\right)\,,
\end{align}
We recall that $\ResSet_{X_f,\twist}$ is given by \eqref{eq:resset-ends} and for one funnel end, the resonances are given by \eqref{eq:resonances_funnel}.
As in the untwisted case (see~\cite[Proposition~2.14]{GuZw97}) we prove the 
following result. 

\begin{prop}\label{prop:tau-factorization}
The relative scattering determinant admits a factorization
\begin{align}\label{eq:tau-factorization}
\tau_{X,\twist}(s) = e^{q(s)} \frac{\ProdTwist(1-s)}{\ProdTwist(s)} \frac{\ProdModel(s)}{\ProdModel(1-s)}\,,
\end{align}
where $q\colon \C \to \C$ is an entire function.
\end{prop}

\begin{proof}
We set 
\begin{equation}\label{eq:auxmap_prodquot}
 h(s) \coloneqq \frac{\ProdTwist(1-s)}{\ProdTwist(s)} \frac{\ProdModel(s)}{\ProdModel(1-s)}
\end{equation}
for any~$s\in\C$, for which the map on the right hand side is defined.
Then $h$ is meromorphic on all of~$\C$, as is the map~$\tau_{X,\twist}$.
It suffices to show that the zeros and poles of the two maps~$h$ and~$\tau_{X,\twist}$ coincide, including their multiplicities.
We first consider $s\in\C$ with $\Rea s = 1/2$. If $s$ is a resonance of~$X$ (or~$X_f$), and hence contributes to the divisor of some of the Weierstrass products in~\eqref{eq:auxmap_prodquot},
then also $1-s$ is a resonance of~$X$ (or~$X_f$, respectively) with the same multiplicity as~$s$.
Therefore the total contribution of~$s$ to the divisor of the quotient of the Weierstrass functions in~\eqref{eq:auxmap_prodquot} cancels.
Thus, $h$ does not have a zero or pole at~$s$.
From~\eqref{eq:sd-unitary} it follows that the same is true for~$\tau_{X,\twist}$.

We consider now $s\in\C$ with~$\Rea s<1/2$ and show that the multiplicities 
of~$s$ as a zero or pole of~$h$ and~$\tau_{X,\twist}$ coincide. Since 
$\tau_{X,\twist}(1-s) = 1/\tau_{X,\twist}(s)$ by~\eqref{eq:sd-functional} as 
well as $h(1-s)=1/h(s)$, this equality of multiplicities then extends 
immediately to the right half plane~$\{\Re s > 1/2\}$. We now pick $\eps > 0$ such 
that the ball of radius $\eps$ around $s$ contains no zeros of the Weierstrass 
products $\ProdTwist$ and $\ProdModel$ except at $s$. Using the argument 
principle, it remains to show that
\begin{equation}\label{eq:argument-principle}
\begin{aligned}
\frac{1}{2\pi i} \int_{\gamma_{s,\eps}} \frac{\tau_{X,\twist}'(t)}{\tau_{X,\twist}(t)} \, dt &= m_{X,\twist}(1-s) - m_{X,\twist}(s)
\\
&\phantom{=} + m_{X_f,\twist}(s) - m_{X_f,\twist}(1-s)\,.
\end{aligned}
\end{equation}
Taking advantage of \eqref{def:scattering_determinant} and 
\eqref{eq:definition_of_m} we can write the left hand side of 
\eqref{eq:argument-principle} as
\begin{align*}
\frac{1}{2\pi i} \int_{\gamma_{s,\eps}} \frac{\tau_{X,\twist}'(t)}{\tau_{X,\twist}(t)} \, dt &= \frac{1}{2\pi i} \int_{\gamma_{s,\eps}} \frac{(\det S_{X,\twist}^\rel(t))'}{\det S_{X,\twist}^\rel(t)} \, dt \\ 
&= \frac{1}{2\pi i} \Tr \int_{\gamma_{s,\eps}}  (S_{X,\twist}^\rel(t))^{-1}  ( 
S_{X,\twist}^\rel(t))'\, dt 
\\ 
& = M_s(S_{X,\twist}^\rel).
\end{align*}
We define the normalized model scattering matrix by
\begin{align*}
\tilde{S}_{X_f,\twist}(s) \coloneqq G(s) \Lambda(s) S_{X_f,\twist}(s) \Lambda(1-s)^{-1} G(1-s)^{-1}
\end{align*}
and obtain, using~\eqref{def:relative_scat_matrix}, that
\begin{align*}
S_{X,\twist}^\rel(s) = G(1-s)^{-1} \Lambda(s) \tilde{S}_{X_f,\twist}(s)^{-1} \tilde{S}_{X,\twist}(s) \Lambda(s)^{-1} G(1-s)\,.
\end{align*}
We recall that $G(1-s)$ and $\Lambda(s)$ are holomorphic for $\Re s < 1/2$.
By \eqref{eq:M-cyclic} we have that
\begin{align*}
M_s(S_{X,\twist}^\rel) = M_s(\tilde{S}_{X,\twist}) - M_s(\tilde{S}_{X_f,\twist})\,.
\end{align*}
Proposition~\ref{prop:sp-res} implies 
\begin{align*}
M_s(\tilde{S}_{X,\twist}) = \nu_{X,\twist}(s) = m_{X,\twist}(1-s) - m_{X,\twist}(s)\,.
\end{align*}
Now $m_{X_f,\twist}(1-s) = 0$ since  $\Re s < 1/2$. We note that the equality 
$M_s(\tilde{S}_{X_f,\twist}) = -m_{X_f,\twist}(s)$ follows directly from
Proposition~\ref{prop:Fell_spole_resonance}, that completes the proof.
\end{proof}

To prove Theorem~\ref{thm:scattering-determinant} we have to show that $q$ is a polynomial of degree at most $4$.
For this, we need a singular value estimate on the relative scattering matrix.
This will give us an estimate on the scattering determinant.
We define the set
\begin{align*}
\mc{L}^0_D \coloneqq \{ s \in \C \setmid D(s) = 0\}\,,
\end{align*}
where $D(s)$, as in \cite[Lemma 6.1]{DFP}, is defined by 
\begin{align*}
D(s) \coloneqq  \det(1 - (L(s) \eta_3)^3)\,. 
\end{align*}
For $\delta>0$ set
\begin{align}\label{eq:B_coll}
\mc B(\delta) \coloneqq B_1(\tfrac12) \cup \bigcup_{\zeta \in \mc{L}^0_D \cup \ResSet_{X_f,\twist} \cup (1- \ResSet_{X_f,\twist})} B_{\ang{\zeta}^{-(2+\delta)}}(\zeta)\,,
\end{align}
where $B_r(z)$ denotes the ball of radius $r$ around $z$.

\begin{lemma}\label{lem:svalues_rel_smatrix}
    For $\delta>0$ large enough there exists $C > 0$ and $c > 0$ such that for $s \not \in \mc B(\delta)$ and $k\in\N$, we have 
    \begin{align*}
        \mu_k(S_{X,\twist}^\rel(s)-\id)) \leq e^{C\ang{s}^{2+\eps}- c k}\,.
    \end{align*}
\end{lemma}
\begin{proof}
    By Proposition~\ref{prop:smatrix-structure}, we have the decomposition
    \begin{align*}
    S_{X,\twist}(s) = S_{X_f,\twist}(s) \oplus 0 + (2s - 1) Q^\#(s)\,.
    \end{align*}
    From \eqref{eq:smatrix-structure} and \eqref{eq:Qhash}, we have that
    \begin{align*}
        S_{X,\twist}(s) &= \begin{pmatrix} S_{X_f,\twist}(s) + Q^\#(s)^{\ff} & S_{X_f,\twist}(s) + Q^\#(s)^{\fc} \\ Q^\#(s)^{\cf} & Q^\#(s)^{\cc} \end{pmatrix}
    \end{align*}
    Hence, the matrix coefficients of $S_{X,\twist}^\rel(s)$ are given by
    \begin{align*}
        S_{X,\twist}^\rel(s) &= \begin{pmatrix} S_\rel^{\ff}(s) & S_\rel^{fc}(s) \\ S_\rel^{\cf}(s) & S_\rel^{\cc}(s) \end{pmatrix} \\
            &= \begin{pmatrix} S_{X_f,\twist}(s)^{-1} & \\ & -\id \end{pmatrix} \begin{pmatrix} S_{X_f,\twist}(s) + Q^\#(s)^{\ff} & Q^\#(s)^{\fc} \\ Q^\#(s)^{\cf} & Q^\#(s)^{\cc} \end{pmatrix} \\
            &= \begin{pmatrix} \id + S_{X_f,\twist}(s)^{-1} Q^\#(s)^{\ff} & S_{X_f,\twist}(s)^{-1} Q^\#(s)^{\fc} \\ -Q^\#(s)^{\cf} & -Q^\#(s)^{\cc} \end{pmatrix}\,.
    \end{align*}
    By \eqref{eq:smatrix-funnel-intertwine}, we have that $S_{X_f,\twist}(s)^{-1} E_{X_f,\twist}(s)^T = - E_{X_f,\twist}(1-s)^T$ and together with \eqref{eq:Qhash}, we obtain
    \begin{align*}
        S_\rel^{\ff}(s) &= \id - (2s - 1)E_{X_f,\twist}(1-s)^T (\eta_3 - \eta_1) \\
        &\phantom{= \id - } \times (\id - L(s) \eta_3)^{-1} [\LapTwist, \eta_0] E_{X_f,\twist}(s)\,,
    \\
    S_\rel^{\fc}(s) &= -(2s-1)E_{X_f,\twist}(1-s)^T (\eta_3 - \eta_1) \\
    &\phantom{= -} \times (\id - L(s) \eta_3)^{-1} [\LapTwist, \eta_0] E_{X_c,\twist}(s)\,,
    \\
    S_\rel^{\cf}(s) &= -(2s-1)E_{X_c,\twist}(s)^T (\eta_3 - \eta_1) (\id - L(s) \eta_3)^{-1} [\LapTwist, \eta_0] E_{X_f,\twist}(s)\,,
    \\
    S_\rel^{\cc}(s) &= -(2s-1)E_{X_c,\twist}(s)^T (\eta_3 - \eta_1) (\id - L(s) \eta_3)^{-1} [\LapTwist, \eta_0] E_{X_c,\twist}(s)\,.
    \end{align*}
    From \eqref{eq:poisson-model-cusp} we obtain for every compactly supported $A \in \Diff^1(X_c, \bundle|_{X_c})$ the bound
    \begin{align}\label{eq:bound-poisson-cusp}
        \norm{A E_{X_c, \twist}(s)} \leq e^{C\ang{s}}
    \end{align}
    for $s \not \in B_1(1/2)$.

    Without loss of generality, we suppose that $X_f$ is a single funnel, that is contained in the hyperbolic cylinder $C_\ell = \ang{h_\ell} \bs \h$.
    If $\Re s > \eps > 0$, we can directly use \eqref{eq:svalues-poisson} to estimate the singular values of $A E_{X_f,\twist}(s)$, where $A \in \Diff^1(X_f,\bundle|_{X_f})$ is compactly supported.
    For $\Re s < 1/2 - \eps$, we use \eqref{eq:svalues-smatrix} and \eqref{eq:svalues-poisson} together with \eqref{eq:smatrix-funnel-intertwine}, $E_{X_f,\twist}(s) S_{X_f,\twist}(s)^{-1} = -E_{X_f,\twist}(1-s)$.
    Hence, for all $s \in \C$, we obtain the estimate
    \begin{align}\label{eq:bound-poisson-funnel}
        \mu_k(A E_{X_f,\twist}(s)) \leq
        \begin{cases}
        d_{k}(s) e^{C \ang{s} \log(s)}, & k \leq \max\{m_0,2m_{\vartheta_j}\}\,, \\
        \left(\frac{\ang{s}}{k}\right)^{2\ang{s}} e^{C\ang{s} - ck}, & k > \max\{m_0,2m_{\vartheta_j}\}\,,
        \end{cases}
    \end{align}
    where $m_\vartheta$ denotes the multiplicity of the eigenvalue $\lambda = 
e^{2\pi i\vartheta}$ of $\twist(h_\ell)$, the function $d_k(s)$ was defined by 
\eqref{eq:def_dks}, and $A \in \Diff^1(X_f, \bundle)$ is compactly supported.
    We note that for every $\delta > 0$ and $s \not \in \mc B(\delta)$, we have that
    \begin{align*}
        d_0(s) \lesssim \ang{s}^{2 + \delta}
    \end{align*}
    due to the fact that there are only finitely many resonances in a ball of 
radius~$1$ around $s$.

    The estimate on the determinant $D(s)$ in \cite[Section~6]{DFP} implies---as 
in the untwisted case (see~\cite[Lemma~3.6]{GuZw97})---that for $\delta > 0$ 
large enough and any
    \[
    s \not \in \bigcup_{\zeta \in \mc{L}^0_D \cup \ResSet_{X_f,\twist}} B_{\ang{\zeta}^{-(2+\delta)}}(\zeta)\,,
    \]
    the following estimate holds for all $\eps > 0$:
    \begin{align}\label{eq:bound-Leta}
    \norm{(\id - L(s) \eta_3)^{-1}}_{L^2(X,\bundle) \to L^2(X,\bundle)} \leq e^{C\ang{s}^{2+\eps}}\,.
    \end{align}
    Using \eqref{eq:bound-poisson-cusp}, \eqref{eq:bound-poisson-funnel}, and 
\eqref{eq:bound-Leta} we obtain 
    \begin{align*}
        \mu_k(S_{X,\twist}^\rel(s)^\bullet) \leq e^{C \ang{s}^{2 + \eps}}\,, \qquad \bullet \in \{\fc, \cf, \cc\}\,,
    \end{align*}
    for $s \not \in \mc B(\delta)$, where we have used that all matrix components involving cusp terms are finite rank operators.
    So in particular, $\mu_k(S_{X,\twist}^\rel(s)^\bullet) = 0$ for $k > N$ for some $N \in \N$.

    For the funnel term, we estimate
    \begin{align*}
        \mu_k( &S_{X,\twist}^\rel(s)^{\ff} - \id ) \\
        &\leq \norm{ (\eta_3 - \eta_1) E_{X_f,\twist}(1-s) } \, \norm{ (\id - L(s) \eta_3)^{-1} } \, \mu_k( [\LapTwist, \eta_0] E_{X_f,\twist}(s) ) \\
        &\leq d_0(1-s) e^{C \ang{1-s} \log(1-s)} \, e^{C\ang{s}^{2+\eps}} \, \mu_k( [\LapTwist, \eta_0] E_{X_f,\twist}(s) ) \,.
    \end{align*}
    From the remark above, we obtain that for $s \not \in \mc B(\delta)$,
    \begin{align*}
        \mu_k( S_{X,\twist}^\rel(s)^{\ff} - \id ) \leq e^{C\ang{s}^{2+\eps}} \, \mu_k( [\LapTwist, \eta_0] E_{X_f,\twist}(s) ) \,.
    \end{align*}
    If $k \leq \max\{m_0, 2m_\varphi\}$, then can apply the same argument to obtain that
    \begin{align*}
        \mu_k( S_{X,\twist}^\rel(s)^{\ff} - \id ) \leq e^{C\ang{s}^{2+\eps}}\,.
    \end{align*}
    For $k \geq \max\{m_0, 2m_\varphi\}$, we have that
    \begin{align*}
        \mu_k( S_{X,\twist}^\rel(s)^{\ff} - \id ) &\leq e^{C\ang{s}^{2+\eps}} \left( \frac{\ang{s}}{k}\right)^{2 \ang{s}} e^{C\ang{s} - ck} \\
        &\leq e^{C\ang{s}^{2+\eps} - ck}\,.
    \end{align*}
\end{proof}

With all these results at our disposal, the proof of Theorem~\ref{thm:scattering-determinant} is analogous to the corresponding statement in the untwisted setting. For the convenience of the reader, we provide the details.

\begin{proof}[Proof of Theorem~\ref{thm:scattering-determinant}]
    In Proposition~\ref{prop:tau-factorization} we established the factorization
    \begin{align}\label{eq:tau-fact}
        \tau_{X,\twist}(s)\cdot\frac{\ProdTwist(s) \ProdModel(1-s) } {\ProdTwist(1-s) \ProdModel(s) } = e^{q(s)}
    \end{align}
    with $q$ being an entire function.
    It remains to show that $q$ is polynomial with degree bounded by~$4$,
    for which we will take advantage of the Hadamard factorization theorem \cite[8.24]{Titchmarsh_book}.
    To that end we let $\varphi\colon\C\to\C$, 
    \[
     \varphi(s) \coloneqq \tau_{X,\twist}(s)\cdot\frac{\ProdTwist(s) \ProdModel(1-s) } {\ProdTwist(1-s) \ProdModel(s) }\,,
    \]
    denote the function on the left hand side of the equation in~\eqref{eq:tau-fact} and note that $\varphi$ is entire and has no zeros (as $q$ is entire).
    Therefore $e^{q(\cdot)}$ is the (full) Hadamard factorization of~$\varphi$, and hence $q$ is polynomial.
    In order to estimate the degree of~$q$, we now provide a numerical bound on the order of~$\varphi$.

    Let $\delta > 0$ be as in Lemma~\ref{lem:svalues_rel_smatrix} and set $\mc B \coloneqq \mc B(\delta)$,
    where $\mc B(\delta)$ is defined in~\eqref{eq:B_coll}.
    We recall that $\mc B$ encloses all zeros of~$\ProdTwist$ and~$\ProdModel$.
    By \cite[Theorem 2.6.5]{Boas_entire} and the upper bounds on the resonances, \cite[Remark 4.13]{DFP} and \cite[Theorem B]{DFP},
    we see that both Weierstrass products~$\ProdTwist$ and~$\ProdModel$ are of order~$2$.
    In combination with the minimum modulus theorem \cite[8.71]{Titchmarsh_book}
       we obtain that for all~$\eps>0$ we have 
    \[
        \log \abs*{\frac{\ProdTwist(s) \ProdModel(1-s) } {\ProdTwist(1-s) 
        \ProdModel(s) } }  \lesssim_\eps \ang{s}^{2+\eps} \qquad\text{for all $s \not\in 
        \mc B$}\,.
    \]
    We may estimate the scattering determinant~$\tau_{X,\twist}$ using 
    \cite[IV.1.2]{GohbergKrein69} and Lemma~\ref{lem:svalues_rel_smatrix} to obtain, 
    for all $s\notin\mc B$,
    \begin{align*}
    \abs{\tau_{X,\twist}(s)} & = \abs{\det ( \id + (S_{X,\twist}^\rel(s)-\id)} 
    \\
    & \leq \prod_{k=1}^\infty \left(1 + \mu_k(S_{X,\twist}^\rel(s)-\id)\right)
    \\
    & \leq \prod_{k=1}^\infty \left( 1 + e^{C\ang{s}^{2+\eps} - ck} \right)
    \end{align*}
    for all~$\eps>0$ and suitable $c, C >0$ (possibly depending on~$\eps$).
    Choose $N(s) \in \N$ such that $c N(s) < C \ang{s}^{2 + \eps} < c (N(s) + 
1)$.
    We have that
    \begin{align*}
        \log \abs{\tau_{X,\twist}(s)} &\leq \sum_{k=1}^\infty \log \left( 1 + e^{C \ang{s}^{2+\eps} - ck} \right) \\
        &=\sum_{k = 1}^{N(s)} \log \left( 1 + e^{C\ang{s}^{2+\eps} - ck} \right) + \sum_{k=N(s)+1}^\infty \log \left( 1 + e^{C\ang{s}^{2+\eps} - ck} \right) \\
        &\lesssim_\eps N(s) \ang{s}^{2 + \eps} + \sum_{j=0}^\infty \log \left( 1 + e^{-c j} e^{C\ang{s}^{2+\eps} - c(N(s) + 1)} \right) \\
        &\lesssim_\eps \ang{s}^{4 + 2\eps}
    \end{align*}
    Therefore, for every~$\eps>0$ and $s \not \in \mc B$, we obtain $C>0$ such 
that 
    \begin{align}\label{eq:bound_on_B}
        \log \abs{\varphi(s)} \lesssim_\eps \ang{s}^{4+\eps}\,.
    \end{align}
        By \cite[Theorem B and Proposition 6.2]{DFP}, we have that
    \begin{align*}
        \#\{ \zeta \in \mc{L}^0_D \cup \ResSet_{X_f,\twist} \cup (1- \ResSet_{X_f,\twist}) \colon \abs{\zeta} \in (r-1, r)\} \lesssim r^2
    \end{align*}
    for any $r > 1$. Hence, we can estimate the area of $\mc B$ restricted to the annulus $\{r-1 < \abs{z} < r\}$ by
    \begin{align*}
        \vol( \mc B \cap \{ z \in \C \colon \abs{z} \in (r-1,r)\} ) &\lesssim_\eps r^2 \ang{r-1}^{-2 (\delta + 2)} \\
        &= O(\ang{r}^{-2 \delta - 2})\,, \quad \text{ as } r \to \infty\,.
    \end{align*}
        Hence, taking $R > 1$ large enough, for any $r > R$ and $s \in \C$ with 
$\abs{s} \leq r$, we have the estimate
    \begin{align*}
        \log \abs{\varphi(s)} \lesssim_\eps \ang{r}^{4+\eps}
    \end{align*}
    by the maximum modulus principle (see for instance \cite[5.1]{Titchmarsh_book}).
    Thus, $\varphi$ is of order~$4$ and hence $q$ is a polynomial of degree at 
most~$4$.
\end{proof}